\documentclass[10pt]{amsart} 
\usepackage{amsmath,amssymb,bm, color,float}

\usepackage{amsmath}
\usepackage{graphicx}

\usepackage{mathrsfs}
\usepackage{bbm}
\usepackage{bm}
\usepackage{amsfonts,amssymb}
\usepackage{multirow}
\usepackage{lineno}
\usepackage{color}

\numberwithin{equation}{section}
 \newtheorem{theorem}{Theorem}[section]
 \newtheorem{lemma}[theorem]{Lemma}

\def\3bar{{|\hspace{-.02in}|\hspace{-.02in}|}}

\def\T{{\mathcal{T}}}

\def\btau{\boldsymbol{\tau}}
\def\beta{\boldsymbol{\eta}}
\def\bgamma{\boldsymbol{\gamma}}

\def\cal#1{{\mathcal #1}}
\def\pT{{\partial T}}

\def\bw{{\mathbf{w}}}
\def\bq{{\mathbf{q}}}

\def\bQ{{\mathbf{Q}}}
\def\bn{{\mathbf{n}}}

\def\beta{{\boldsymbol{\eta}}}
\def\btheta{{\boldsymbol{\theta}}}

\def\btau{{\boldsymbol{\tau}}}

\newtheorem{remark}{Remark}[section]
\newtheorem{algorithm}{Weak Galerkin Algorithm}[section]

\numberwithin{equation}{section}

\def\3bar{{|\hspace{-.02in}|\hspace{-.02in}|}}
 
 \def\cal#1{\mathcal{#1}}
\def\ad#1{\begin{aligned}#1\end{aligned}}   
\def\a#1{\begin{align*}#1\end{align*}} \def\an#1{\begin{align}#1\end{align}}

\begin{document}

\title []
 {A Simple Weak Galerkin Finite Element Method for the Reissner-Mindlin Plate Model on Non-Convex Polytopal Meshes}

  \author {Chunmei Wang} 
  \address{Department of Mathematics, University of Florida, Gainesville, FL 32611, USA. }
  \email{chunmei.wang@ufl.edu}

\author {Shangyou Zhang}
\address{Department of Mathematical Sciences,  University of Delaware, Newark, DE 19716, USA}   \email{szhang@udel.edu}

\begin{abstract}
 This paper presents a simple weak Galerkin (WG) finite element method for the Reissner–Mindlin plate model that partially eliminates the need for traditionally employed stabilizers. The proposed approach accommodates general, including non-convex, polytopal meshes, thereby offering greater geometric flexibility. It utilizes bubble functions without imposing the restrictive conditions required by existing stabilizer-free WG methods, which simplifies implementation and broadens applicability to a wide range of partial differential equations (PDEs). Moreover, the method allows for flexible choices of polynomial degrees in the discretization and can be applied in any spatial dimension. We establish optimal-order error estimates for the WG approximation in a discrete 
$H^1$
  norm, and present numerical experiments that validate the theoretical results.
\end{abstract}


 
\maketitle

\section{Introduction}
In this paper, we present a simple weak Galerkin finite element method on non-convex polytopal meshes. The method is applied to the Reissner–Mindlin plate model with clamped boundary conditions, seeking a rotation field $\btheta$  and a transverse displacement $w$ satisfying: 
\an{ \label{m1}
    \ad{
-\nabla \cdot ({\cal C}\epsilon (\btheta))-\lambda t^{-2} (\nabla w -\btheta)=&0, \qquad \text{in} \ \Omega,\\
-\nabla \cdot \lambda t^{-2}(\nabla w -\btheta)=&g, \qquad \text{in} \ \Omega,\\
 \btheta =0, \qquad w=&0,  \qquad \text{on} \ \partial\Omega, } }
where $g\in L^2(\Omega)$ is the transverse load, $C$  denotes the tensor of bending moduli, $\lambda$ is the shear correction factor, $t$ is the plate thickness, and   $\epsilon (\btheta)=\frac{1}{2}(\nabla\btheta+\nabla \btheta^T)$ is the symmetric gradient of $\btheta$. The computational domain   $\Omega\subset \mathbb R^d$ is assumed to be open, bounded, and with a Lipschitz continuous boundary $\partial\Omega$. While the formulation holds for general 
$d$-dimensional domains, for clarity of exposition we present the method and analysis in the two-dimensional case ($d=2$).
Note that all subsequent analysis can be readily extended to any spatial dimension 
$d$ without difficulty.

A typical weak formulation of \eqref{m1} seeks
 $(w,\btheta)\in H_0^1(\Omega)\times [H_0^1(\Omega)]^2$, such that
\begin{equation}\label{model} 
({\cal C}\epsilon(\btheta), \epsilon(\eta))+\lambda t^{-2} (\nabla w-\btheta, \nabla v-\eta)  =  (g, v), \qquad\qquad \text{in}\quad \Omega, 
 \end{equation}
for any $(v, \beta)\in H_0^1(\Omega)\times [H_0^1(\Omega)]^2$.

Shear locking remains a major obstacle in the numerical treatment of plate models. Achieving a locking-free discretization—one that converges uniformly with respect to the plate thickness 
$t$—is a nontrivial task \cite{1}. Established remedies include projection-based techniques, reduced integration, and mixed formulations \cite{1,2,3,4,5,6,7,8,9,10,11,12,13,14,15,16,17,18}. Locking-free methods without reduction operators have also been proposed in non-primal settings, notably within mixed, discontinuous Galerkin, and virtual element frameworks \cite{19,20,21,22,23,24,25}.
Finite element methods on general polygonal meshes have advanced rapidly in recent years due to their flexibility and broad applicability, with significant developments for second-order elliptic problems \cite{26,27,28,29}.  
 
The weak Galerkin (WG) finite element method represents a notable advancement in the numerical approximation of partial differential equations (PDEs). Its core idea is to reinterpret or approximate differential operators in a manner reminiscent of distribution theory, but tailored to piecewise polynomial spaces. Unlike traditional finite element techniques, WG relaxes regularity requirements on trial and test functions by introducing carefully designed stabilizers. Over the past decade, the WG framework has been extensively studied for a broad spectrum of model PDEs \cite{wg1, wg2, wg3, wg4, wg5, wg6, wg7, wg8, wg9, wg10, wg11, wg12, wg13, wg14, wg15, wg16, wg17, wg18, wg19, wg20, wg21, itera, wy3655, he, ku, zhang, yang, wang, li}, establishing its versatility and robustness in scientific computing. What distinguishes WG from other finite element approaches is its systematic use of weak derivatives and weak continuity concepts to construct schemes directly from the weak form of the governing PDEs. This flexibility has enabled the successful development of WG schemes for the Reissner–Mindlin plate model, albeit with stabilizers \cite{RM}.

A noteworthy extension of the WG framework is the Primal–Dual Weak Galerkin (PDWG) method \cite{pdwg1, pdwg2, pdwg3, pdwg4, pdwg5, pdwg6, pdwg7, pdwg8, pdwg9, pdwg10, pdwg11, pdwg12, pdwg13, pdwg14, pdwg15}. PDWG reformulates the discrete problem as a constrained optimization problem, where the constraints encode the weak formulation via weak derivatives. The resulting Euler–Lagrange equations couple the primary unknowns with dual variables (Lagrange multipliers), producing a symmetric numerical scheme. This approach effectively addresses several challenges that are difficult to handle using more conventional numerical methods.

This paper presents a simple weak Galerkin finite element formulation for the Reissner–Mindlin plate model that applies to both convex and non-convex polytopal meshes, eliminating  the need for one stabilizer. The elimination of one stabilizer  is achieved by employing higher-order polynomials in the computation of the discrete weak gradient. In contrast to existing stabilizer-free WG methods that are restricted to convex elements \cite{ye}, the proposed approach handles general non-convex polytopal meshes while maintaining the size and global sparsity pattern of the stiffness matrix, and at the same time substantially simplifying implementation. A rigorous analysis is provided, establishing optimal error estimates for the WG approximations in a discrete 
$H^1$ norm. Numerical experiments are included to confirm the theoretical results.

 In Section 2, we briefly review the definition of the weak gradient  and its discrete version.
In Section 3, we present the simple weak Galerkin scheme without the use of one stabilizer.
Section 4 is dedicated to deriving the existence and uniqueness of the solution.
In Section 5, we derive the error equation for the proposed weak Galerkin scheme.
Section 6 focuses on deriving the error estimate for the numerical approximation in the energy norm.
Finally, numerical results are provided to demonstrated our theoretical findings in Section 7.  

The standard notations are adopted throughout this paper. Let $D$ be any open bounded domain with Lipschitz continuous boundary in $\mathbb{R}^d$. We use $(\cdot,\cdot)_{s,D}$, $|\cdot|_{s,D}$ and $\|\cdot\|_{s,D}$ to denote the inner product, semi-norm and norm in the Sobolev space $H^s(D)$ for any integer $s\geq0$, respectively. For simplicity, the subscript $D$ is  dropped from the notations of the inner product and norm when the domain $D$ is chosen as $D=\Omega$. For the case of $s=0$, the notations $(\cdot,\cdot)_{0,D}$, $|\cdot|_{0,D}$ and $\|\cdot\|_{0,D}$ are simplified as $(\cdot,\cdot)_D$, $|\cdot|_D$ and $\|\cdot\|_D$, respectively.

\section{Weak Gradient Operator and Discrete Weak Gradient Operator}\label{Section:Hessian}
In this section, we briefly review the definition of the weak gradient operator and its discrete counterpart, following \cite{ye, RM}.

Let $T$ be a polygonal element with boundary $\partial T$.
A scalar-valued weak function on $T$ is defined as
  $w=\{w_0, w_b\}$ where $w_0 \in L^2(T)$ and $w_b \in L^2(\partial T)$.  
The component $w_0$ represents the value of $w$ in the interior of $T$, while $w_b$ represents its value on $\partial T$.  
In general, $w_b$ is assumed to be independent of the trace of $w_0$.
  A vector-valued weak function on $T$ is defined as
 $\beta=\{\beta_0, \beta_b\}$ where $\beta_0 \in [L^2(T)]^2$ and $\beta_b \in [L^2(\partial T)]^2$.  
Here, $\beta_0$ denotes the value of $\boldsymbol{\beta}$ in the interior of $T$, and $\beta_b$ denotes its value on $\partial T$.  
Similarly, $\beta_b$ is not necessarily the trace of $\beta_0$.

 The space of all scalar-valued weak functions on $T$ is given by
\begin{equation}\label{2.1}
 W (T)=\{w=\{w_0,w_b\}: w_0\in L^2(T), w_b\in L^{2}(\partial
 T)\}.
\end{equation}
The space of all vector-valued weak functions on $T$ is given by
\begin{equation}\label{2.1-1}
\Theta (T)=\{\beta=\{\beta_0,\beta_b\}: \beta_0\in [L^2(T)]^2, \beta_b\in [L^{2}(\partial
 T)]^2\}.
\end{equation}

Let $\boldsymbol{n}$ denote the unit outward normal vector to $\partial T$.  
For any $w \in W(T)$, the \emph{weak gradient} $\nabla_w w$ is defined as the bounded linear functional on $[H_0^1(T)]^2$ satisfying
\begin{equation}
        (\nabla_w w, \bq)_T=-(w_0, \nabla\cdot \bq)_T+\langle w_b, \bq \cdot \bn\rangle_{\partial T}, \qquad \forall\bq \in [H_0^1(T)]^2.
\end{equation}

For any $\beta\in \Theta (T)$, the weak symmetric gradient $\epsilon_w  (\beta) \in  [H_0^1(T)]_{sym}^{2\times 2}$ is defined on each element $T$ by
 \begin{equation}\label{2.3-first}
  (\epsilon_w (\beta), \btau)_T=-(\beta_0, \nabla \cdot \btau)_T+
  \langle \beta_b, \btau \cdot \bn \rangle_{\partial T},\quad \forall \btau\in   [H_0^1(T)]_{sym}^{2\times 2}.
  \end{equation}
  
Let $P_r(T)$ denote the space of polynomials on $T$ of total degree no greater than $r$, where $r \ge 0$ is an integer.  
 The discrete weak gradient operator on $T$, denoted   $\nabla_{w, r_1, T}$, is a linear operator
 from $W(T)$ to $[P_{r_1}(T)]^2$, such that,  for any $w\in W(T)$,
 $\nabla_{w, r_1, T}w$ is the unique polynomial vector in $[P_{r_1}(T)]^2$ satisfying 
 \begin{equation}\label{2.4}
 (\nabla_{w, r_1, T} w, \bq)_T=-(w_0, \nabla \cdot\bq)_T+
  \langle w_b, \bq\cdot \bn \rangle_{\partial T},\quad \forall \bq \in [P_{r_1}(T)]^2.
  \end{equation}
If $w_0\in
 H^1(T)$, applying the standard integration by parts to the first term on the right-hand side of \eqref{2.4} yields
\begin{equation}\label{2.4new}
   (\nabla_{w, r_1, T} w, \bq)_T= (\nabla w_0,  \bq)_T+
  \langle w_b-w_0, \bq\cdot \bn \rangle_{\partial T},\quad \forall \bq \in [P_{r_1}(T)]^2.
  \end{equation}  

Similarly, the \emph{discrete weak symmetric gradient} on $T$, denoted $\epsilon_{w,r_2,T}$, is a linear mapping
\[
\epsilon_{w,r_2,T} : \Theta(T) \longrightarrow [P_{r_2}(T)]^{2 \times 2}_{\mathrm{sym}},
\]
such that, for any $\boldsymbol{\beta} \in \Theta(T)$,
 \begin{equation}\label{2.3}
  (\epsilon_w (\beta), \btau)_T=-(\beta_0, \nabla \cdot \btau)_T+
  \langle \beta_b, \btau \cdot \bn \rangle_{\partial T},\quad \forall \btau\in   [P_{r_2}(T)]_{sym}^{2\times 2}.
  \end{equation}
If $\beta_0\in
 [H^1(T)]^2$, using integration by parts in \eqref{2.3} gives
\begin{equation}\label{2.3new}
  (\epsilon_w (\beta), \btau)_T= (\epsilon(\beta_0),   \btau)_T+
  \langle \beta_b-\beta_0, \btau \cdot \bn \rangle_{\partial T},\quad \forall \btau\in   [P_{r_2}(T)]_{sym}^{2\times 2}.
  \end{equation}
 
\section{Weak Galerkin Algorithms}\label{Section:WGFEM}
Let ${\cal T}_h$ be a shape-regular finite element partition of the polygonal domain $\Omega \subset \mathbb{R}^2$ \cite{wy3655}.
Denote by ${\mathcal E}_h$ the set of all edges in ${\cal T}_h$, and by ${\mathcal E}_h^0 = {\mathcal E}_h \setminus \partial\Omega$ the set of interior edges.
For each $T \in {\cal T}_h$, let $h_T$ be its diameter, and set $h = \max{T\in {\cal T}_h} h_T$.

Let $k,p \ge 1$ be integers.
For $T \in {\cal T}_h$, define the local weak finite element space
 \an{\label{W-T}
 W(k, p,  T) & =\{\{v_0,v_b \}: v_0\in P_k(T), v_b\in P_p(e),  e\subset \partial T\}.    
   }
By enforcing continuity of $v_b$ on $\mathcal{E}_h^0$, we obtain the global space
 $$
 W_h=\big\{\{v_0,v_b \}:\ \{v_0,v_b \}|_T\in W(k, p,   T),
 \forall T\in {\cal T}_h \big\}.
 $$
The subspace with vanishing boundary values is
$$
W_h^0=\{\{v_0,v_b\}\in W_h: v_b|_{e}=0, e\subset\partial\Omega\}.
$$

 Similarly, let $q,t \ge 1$ be integers.
For $T \in {\cal T}_h$, define
 \an{\label{t-T}
\Theta(q, t,  T) &=\{\{\beta_0,\beta_b \}: \beta_0\in [P_q(T)]^2, \beta_b\in [P_t(e)]^2,  e\subset \partial T\}.   
    }
Patching $\beta_b$ continuously across $\mathcal{E}_h^0$ yields 
 $$
 \Theta_h=\big\{\{\beta_0,\beta_b \}:\ \{\beta_0,\beta_b \}|_T\in \Theta(q, t, T),
 \forall T\in {\cal T}_h \big\},
 $$
with homogeneous boundary subspace
$$
\Theta_h^0=\{\{\beta_0,\beta_b\}\in \Theta_h: \beta_b|_{e}=0, e\subset\partial\Omega\}.
$$

For  $w\in W_h$, the discrete weak gradient is
$$
(\nabla_{w} w)|_T= \nabla_{w, r_1, T}(w |_T), \qquad \forall T\in \T_h,
$$
where $\nabla_{w,r_1,T}$ is defined by \eqref{2.4}.

 For $\beta\in \Theta_h$,  the discrete weak symmetric gradient is
$$
(\epsilon_{w}  \beta)|_T= \epsilon_{w, r_2, T}(\beta |_T), \qquad \forall T\in \T_h,
$$
with $\epsilon_{w,r_2,T}$ given by \eqref{2.3}.
 
  On each element $T\in\T_h$, let $Q_0$ be the $L^2$ projection onto $P_k(T)$ and $Q_b$ be the $L^2$ projection operator onto $P_{p}(e)$ for each $e \subset \partial T$. 
 For  $w\in H^1(\Omega)$, the $L^2$ projection into $W_h$ is defined by
 $$
  (Q_hw)|_T:=\{Q_0(w|_T),Q_b(w|_{\pT}) \},\qquad \forall T\in\T_h.
$$

Similarly, on $T \in {\cal T}_h$, let $\bQ_0$ be the $L^2$ projection onto $[P_q(T)]^2$ and $\bQ_b$ be the $L^2$ projection operator onto $[P_{t}(e)]^2$ for $e \subset \partial T$. 
 For   $\beta\in [H^1(\Omega)]^2$, the $L^2$ projection into $\Theta_h$ is given by
 $$
  (\bQ_h \beta)|_T:=\{\bQ_0(\beta|_T), \bQ_b(\beta|_{\pT}) \},\qquad \forall T\in\T_h.
$$

Let ${\cal Q}_{r_1}$ and ${\cal Q}_{r_2}$ be the $L^2$ projections onto  
$[P_{r_1}(T)]^2$ and   $[P_{r_2}(T)]^{2\times 2}$, respectively.

Using the weak formulation \eqref{model} of \eqref{m1}, the  simple weak Galerkin scheme is as follows.
\begin{algorithm}\label{PDWG1}
Find $\btheta_h=\{\btheta_0, \btheta_b\}\in \Theta_h^0$, and 
$ w_h=\{w_0, w_b \} \in W_h^0$ such that 
\begin{equation}\label{WG}
({\cal C}\epsilon_w (\btheta_h), \epsilon_w(\beta))+\lambda t^{-2}
   b((w_h,\btheta_0),(v,\beta_0))+s(w_h,v)=(g, v_0), 
\end{equation}
  for all $v=\{v_0, v_b\}\in W_h^0$ and $\beta=\{\beta_0,\beta_b\} \in \Theta_h^0$,
where  
\a{ (\cdot, \cdot)&=\sum_{T\in {\cal T}_h}  (\cdot, \cdot)_T, \\
    b((w_h,\btheta_0),(v,\beta_0)) &=
    (\nabla_{w} w_h-{\cal Q}_{r_1}\btheta_0, \nabla_{w} v-{\cal Q}_{r_1}\beta_0)\\
   s(w_h,v) &=\sum_{T\in {\cal T}_h}  h^{-1} \langle w_0-w_b, v_0-v_b\rangle_{\partial T} . }
\end{algorithm}

\section{Solution Existence and Uniqueness} 
Recall that ${\cal T}_h$ is a shape-regular finite element partition of $\Omega$.
For any $T \in {\cal T}_h$ and $\phi \in H^1(T)$, the trace inequality \cite{wy3655} holds:
\begin{equation}\label{tracein}
 \|\phi\|^2_{\partial T} \leq C(h_T^{-1}\|\phi\|_T^2+h_T \|\nabla \phi\|_T^2).
\end{equation}
If $\phi$ is a polynomial on $T$, the following trace inequality \cite{wy3655} holds:
\begin{equation}\label{trace}
\|\phi\|^2_{\partial T} \leq Ch_T^{-1}\|\phi\|_T^2.
\end{equation}

For $v=\{v_0, v_b\}\in W_h$,   define  \begin{equation}\label{3norm2}
\3bar v\3bar_{W_h}= (\nabla_{w}v, \nabla_{w}v) ^{\frac{1}{2}}, 
\end{equation}
and the discrete $H^1$ semi-norm
\begin{equation}\label{disnorm2}
\|v\|_{1, h}=\Big( \sum_{T\in {\cal T}_h} \|\nabla v_0\|_T^2+h_T^{-1}\|v_0-v_b\|_{\partial T}^2\Big)^{\frac{1}{2}}.
\end{equation}

For $\beta=\{\beta_0, \beta_b\}\in \Theta_h$,  define
\begin{equation}\label{3norm}
\3bar \beta\3bar_{\Theta_h}= (\epsilon_{w}\beta, \epsilon_{w}\beta) ^{\frac{1}{2}},
\end{equation}
and 
\begin{equation}\label{disnorm}
\|\beta\|_{1, h}=\Big( \sum_{T\in {\cal T}_h} \|\epsilon \beta_0\|_T^2+h_T^{-1}\|\beta_0-\beta_b\|_{\partial T}^2\Big)^{\frac{1}{2}}.
\end{equation}

\begin{lemma}\cite{wang}\label{norm1}
 For $w=\{w_0, w_b \}\in W_h$, there exists a constant $C>0$ such that
 $$
 \|\nabla w_0\|_T\leq C\|\nabla_{w} w\|_T.
 $$
\end{lemma} 

\begin{remark}  \cite{wang}
 For nonconvex element, we take  $r_1=2N+k-1$ and $r_2=2N+k-1$.  For convex element,  we take $r_1=N+k-1$ and $r_2=N+k-1$. Here $N$ denotes the number of edges of the polygons.  
\end{remark}

 \begin{lemma}\label{norm2}\cite{ela}
 For $\beta=\{\beta_0, \beta_b \}\in \Theta_h$, there exists a constant $C>0$ such that
 $$
 \|\epsilon \beta_0\|_T\leq C\|\epsilon_{w} \beta\|_T.
 $$
\end{lemma}  

\begin{lemma}\label{normeqva} \cite{wang}  There exist $C_1, C_2 > 0$ such that for any  $v=\{w_0, w_b \} \in W_h$, 
 \begin{equation}\label{normeq}
 C_1\|w\|_{1, h}\leq \3bar v\3bar_{W_h}  \leq C_2\|w\|_{1, h}.
\end{equation}
\end{lemma}

\begin{lemma}\label{normeqva2}  \cite{ela} There exist $C_1, C_2 > 0$ such that for any $\beta=\{\beta_0, \beta_b \} \in \Theta_h$,  
 \begin{equation}\label{normeq2}
 C_1\|\beta\|_{1, h}\leq \3bar \beta\3bar_{\Theta_h}  \leq C_2\|\beta\|_{1, h}.
\end{equation}
\end{lemma}  

\begin{theorem}
The  WG Algorithm \ref{PDWG1} admits a unique solution.
\end{theorem}
\begin{proof}
It suffices to show that if $g=0$, the only solution is the trivial one.  Setting  $(\beta, v)=(\btheta_h, w_h)$ in \eqref{WG} with $g=0$ yields 
$$
({\cal C}\epsilon_w  (\btheta_h), \epsilon_w(\btheta_h))+\lambda t^{-2}
(\nabla_{w} w_h- {\cal Q}_{r_1}\btheta_0, \nabla_{w} w_h- {\cal Q}_{r_1}\btheta_0)+s(w_h, w_h)=0.
$$
Hence $({\cal C}\epsilon_w  (\btheta_h), \epsilon_w(\btheta_h))=0$,   $\lambda t^{-2}
(\nabla_{w} w_h-{\cal Q}_{r_1} \btheta_0, \nabla_{w} w_h- {\cal Q}_{r_1}\btheta_0)=0$ and $s(w_h, w_h)=0$. 

From $({\cal C}\epsilon_w  (\btheta_h), \epsilon_w(\btheta_h))=0$ and \eqref{normeq2}, we obtain 
$$
\|\btheta_h\|_{1, h}=0,
$$
which implies  $\epsilon \btheta_0=0$ in each $T\in {\cal T}_h$ and $\btheta_0=\btheta_b$ on each $\partial T$. Thus  $\btheta_0$ is a constant in $\Omega$. Since $\btheta_0=\btheta_b$ on each $\partial T$ and $\btheta_b=0$ on $\partial\Omega$, it follows that $\btheta_0\equiv 0$ in $\Omega$, hence $\btheta_b\equiv 0$ and $\btheta_h\equiv 0$ in $\Omega$.  Substituting $\btheta_0\equiv 0$ in $\Omega$ into   $\lambda t^{-2}
(\nabla_{w} w_h- {\cal Q}_{r_1} \btheta_0, \nabla_{w} w_h- {\cal Q}_{r_1} \btheta_0)=0$ yields 
$\lambda t^{-2}
(\nabla_{w} w_h , \nabla_{w} w_h)=0$. This, from \eqref{normeq},  gives
$$\|w_h\|_{1,h}=0.$$ This implies $\nabla w_0=0$ on each $T\in {\cal T}_h$ and $w_0=w_b$ on each $\partial T$. This indicates $w_0$ is a constant in the domain $\Omega$. Using $w_0=w_b$ on each $\partial T$ and $w_b=0$ on $\partial\Omega$, we have $w_0\equiv 0$ in $\Omega$. Thus, we have $w_b\equiv 0$ and $w_h\equiv 0$ in $\Omega$.

Therefore, the solution is trivial, completing the proof. 
\end{proof}

\section{Error Equations} 
In this section, we derive the error equations for the WG scheme \ref{PDWG1}.

\begin{lemma}\label{Lemma5.1}   The following identities hold for any element $T$:
\begin{equation}\label{pro}
\nabla_{w}u ={\cal Q}_{r_1}(\nabla u), \qquad \forall u\in H^1(T),
\end{equation}
\begin{equation}\label{pro2}
\epsilon_{w} \beta ={\cal Q}_{r_2}(\epsilon \beta), \qquad \forall \beta\in H^1(T).
\end{equation}
\begin{equation}\label{pro3}
\nabla_{w}Q_hu ={\cal Q}_{r_1}(\nabla u), \qquad \forall u\in H^1(T).
\end{equation} 
\end{lemma}

\begin{proof} Let $u \in H^1(T)$. By the definition in \eqref{2.4new}, we have
 \begin{equation*} 
  \begin{split}
 &(\nabla_{w}u, \bw)_T\\
  =& (\nabla u,  \bw)_T+
  \langle u|_{\partial T}-u|_T, \bw   \cdot\bn \rangle_{\partial T} \\
  =& (\nabla u,  \bw)_T=({\cal Q}_{r_1}(\nabla u),  \bw)_T,
   \end{split}
   \end{equation*} 
  for any $\bw\in [P_{r_1}(T)]^2$.  
  
  The results in \eqref{pro2} and \eqref{pro3} follow from analogous arguments.

This completes the proof.
  \end{proof}

Let $w$ and $\btheta$ denote the exact solution of the Reissner–Mindlin plate model \eqref{m1}, and let $w_h \in W_{h}$ and $\btheta_h \in \Theta_h$ be the corresponding numerical approximations obtained from the Weak Galerkin scheme \ref{PDWG1}.  
For convenience, introduce the shear stress
 $\bgamma=\lambda t^{-2} (\nabla w-\btheta)$ and its discrete approximation $\bgamma_h=\lambda  t^{-2} (\nabla_w w_h-{\cal Q}_{r_1}\btheta_0)$.

We define the error functions associated with the shear stress and the rotation as
\begin{equation}\label{error} 
e_{\bgamma_h}=\bgamma-\bgamma_h.
\end{equation}
\begin{equation}\label{error2}
e_{\btheta_h}=\btheta-\btheta_h.
\end{equation}

\begin{lemma}\label{errorequa}
The error functions $e_{\bgamma_h}$ and $e_{\btheta_h}$ defined in \eqref{error}–\eqref{error2} satisfy the following error equation:
 \begin{equation}\label{erroreqn}
    \begin{split}
    &\sum_{T\in {\cal T}_h} ({\cal C} \epsilon_w (e_{\btheta_h}), \epsilon_w (\beta))_T  +(  e_{\bgamma_h}, \nabla_w v- {\cal Q}_{r_1} \beta_0)_T-s(w_h, v)\\=&  \sum_{T\in {\cal T}_h}  \langle {\cal C}(\epsilon (\btheta)-   {\cal Q}_{r_2} (\epsilon  \btheta))\cdot\bn, \beta_0-\beta_b\rangle_{\partial T}\\
    &+\langle v_0-v_b, (\bgamma -{\cal Q}_{r_1} \bgamma) \cdot \bn\rangle_{\partial T},
    \end{split}
\end{equation} 
for all $\beta \in \Theta_h^0$ and $v \in W_h^0$.
\end{lemma}
\begin{proof}   
Testing the first equation in \eqref{m1} with $\beta_0$ of $\beta = \{\beta_0, \beta_b\} \in \Theta_h^0$ yields
\begin{equation}\label{eee}
 \sum_{T\in {\cal T}_h} ({\cal C} \epsilon (\btheta), \epsilon (\beta_0))_T-\langle {\cal C}\epsilon (\btheta)\cdot\bn, \beta_0-\beta_b\rangle_{\partial T} -(\bgamma, \beta_0)_T=0,   
\end{equation}
where we have used the usual integration by parts and the fact that $\sum_{T\in {\cal T}_h} \langle {\cal C}\epsilon (\btheta)\cdot\bn,  \beta_b\rangle_{\partial T}=\langle {\cal C}\epsilon (\btheta)\cdot\bn,  \beta_b\rangle_{\partial \Omega}=0$ since $\beta_b=0$ on $\partial\Omega$.  

For the term $\sum_{T\in {\cal T}_h} ({\cal C} \epsilon (\btheta), \epsilon (\beta_0))_T$, using \eqref{pro2} and \eqref{2.3new} we obtain 
\begin{equation}\label{eee1}
    \begin{split}
    ({\cal C} \epsilon_w \btheta,\epsilon_w \beta)_T=&({\cal C} {\cal Q}_{r_2} (\epsilon  \btheta), \epsilon_w \beta)_T\\
    =& (\epsilon(\beta_0), {\cal C} {\cal Q}_{r_2} (\epsilon  \btheta))_T-\langle \beta_0-\beta_b, {\cal C} {\cal Q}_{r_2} (\epsilon  \btheta))\cdot\bn \rangle_{\partial T}\\
     =& (\epsilon(\beta_0),  {\cal C}\epsilon ( \btheta) )_T-\langle \beta_0-\beta_b, {\cal C} {\cal Q}_{r_2} (\epsilon  \btheta))\cdot\bn \rangle_{\partial T}.
    \end{split}
\end{equation}

Substituting \eqref{eee1} into \eqref{eee} yields
\begin{equation}\label{eee4}
    \sum_{T\in {\cal T}_h} ({\cal C} \epsilon_w (\btheta), \epsilon_w (\beta))_T  -(\bgamma, \beta_0)_T= \sum_{T\in {\cal T}_h}\langle {\cal C}(\epsilon (\btheta)-   {\cal Q}_{r_2} (\epsilon  \btheta))\cdot\bn, \beta_0-\beta_b\rangle_{\partial T}. 
\end{equation}

Next, testing the second equation in \eqref{m1} with $v_0$ of $v = \{v_0, v_b\} \in W_h^0$ gives
\begin{equation}\label{eee2}
   \sum_{T\in {\cal T}_h} -(\nabla \cdot \bgamma, v_0)_T=\sum_{T\in {\cal T}_h} (\bgamma, \nabla v_0)_T-\langle v_0-v_b, \bgamma \cdot \bn\rangle_{\partial T} =\sum_{T\in {\cal T}_h}(g, v_0)_T,
\end{equation}
where we used the usual integration by parts and  the fact that $\sum_{T\in {\cal T}_h} \langle v_b, \bgamma\cdot\bn\rangle_{\partial T} =\langle v_b, \bgamma\cdot\bn\rangle_{\partial \Omega}=0$ since $v_b=0$ on $\partial \Omega$.

From \eqref{2.4new}, we have
\begin{equation}\label{eee3}
    \begin{split}
(\bgamma, \nabla v_0)_T =&({\cal Q}_{r_1} \bgamma,  \nabla v_0)_T\\
=&({\cal Q}_{r_1} \bgamma, \nabla_w v)_T+\langle v_0-v_b, {\cal Q}_{r_1} \bgamma \cdot\bn\rangle_{\partial T}\\
=&(  \bgamma, \nabla_w v)_T+\langle v_0-v_b, {\cal Q}_{r_1} \bgamma \cdot\bn\rangle_{\partial T}.
    \end{split}
\end{equation} 

Substituting \eqref{eee3} into \eqref{eee2} yields
\begin{equation}\label{eee5}
    \sum_{T\in {\cal T}_h} (  \bgamma, \nabla_w v)_T= \sum_{T\in {\cal T}_h} \langle v_0-v_b, (\bgamma -{\cal Q}_{r_1} \bgamma) \cdot \bn\rangle_{\partial T} +(g, v_0)_T.
\end{equation}

Adding \eqref{eee4}
 and \eqref{eee5} gives
 \begin{equation*} 
    \begin{split}
    &\sum_{T\in {\cal T}_h} ({\cal C} \epsilon_w (\btheta), \epsilon_w (\beta))_T  +(  \bgamma, \nabla_w v-\beta_0)_T\\=&  \sum_{T\in {\cal T}_h}  \langle {\cal C}(\epsilon (\btheta)-   {\cal Q}_{r_2} (\epsilon  \btheta))\cdot\bn, \beta_0-\beta_b\rangle_{\partial T}\\
    &+\langle v_0-v_b, (\bgamma -{\cal Q}_{r_1} \bgamma) \cdot \bn\rangle_{\partial T} +(g, v_0)_T.
    \end{split}
\end{equation*} 

Using the property of  ${\cal Q}_{r_1}$ yields
\begin{equation*} 
    \begin{split}
    &\sum_{T\in {\cal T}_h} ({\cal C} \epsilon_w (\btheta), \epsilon_w (\beta))_T  +( {\cal Q}_{r_1} \bgamma, \nabla_w v- {\cal Q}_{r_1} \beta_0)_T\\=&  \sum_{T\in {\cal T}_h}  \langle {\cal C}(\epsilon (\btheta)-   {\cal Q}_{r_2} (\epsilon  \btheta))\cdot\bn, \beta_0-\beta_b\rangle_{\partial T}\\
    &+\langle v_0-v_b, (\bgamma -{\cal Q}_{r_1} \bgamma) \cdot \bn\rangle_{\partial T}  +(g, v_0)_T,
    \end{split}
\end{equation*} 
and further gives
 \begin{equation}\label{eee8}
    \begin{split}
    &\sum_{T\in {\cal T}_h} ({\cal C} \epsilon_w (\btheta), \epsilon_w (\beta))_T  +(   \bgamma, \nabla_w v- {\cal Q}_{r_1} \beta_0)_T-s(w_h, v)\\=&  \sum_{T\in {\cal T}_h}  \langle {\cal C}(\epsilon (\btheta)-   {\cal Q}_{r_2} (\epsilon  \btheta))\cdot\bn, \beta_0-\beta_b\rangle_{\partial T}\\
    &+\langle v_0-v_b, (\bgamma -{\cal Q}_{r_1} \bgamma) \cdot \bn\rangle_{\partial T}  +(g, v_0)_T.
    \end{split}
\end{equation}

 Finally, subtracting \eqref{WG} from \eqref{eee8} yields the desired result \eqref{erroreqn}, completing the proof.

\end{proof}

\section{Error Estimates}
  In this section, we establish an error estimate  in a discrete $H^1$-norm for the WG solutions of WG Scheme~\ref{PDWG1}.

\begin{lemma}\label{w}
Let $w$ be the exact solution of the RM plate model \eqref{m1}, and suppose $w \in H^{k+1}(\Omega)$. Then there exists a constant $C>0$ such that
\begin{equation}\label{erroresti1}
\3bar w-Q_hw \3bar_{W_h} \leq Ch^{k}\|w\|_{k+1}.
\end{equation}
\end{lemma}
\begin{proof}

Using \eqref{2.4new}, the Cauchy-Schwarz inequality, and the trace inequalities \eqref{tracein}--\eqref{trace}, for any $\bq \in [P_{r_1}(T)]^2$ we have

\begin{equation*}
\begin{split}
&\sum_{T\in {\cal T}_h} (\nabla_{w}(w-Q_hw),  \bq)_T\\
 = &\sum_{T\in {\cal T}_h} (\nabla(w-Q_0w),  \bq)_T+
  \langle Q_0w-Q_bw,  \bq \cdot\bn \rangle_{\partial T}\\
\leq &\Big(\sum_{T\in {\cal T}_h} \|\nabla(w-Q_0w)\|^2_T\Big)^{\frac{1}{2}} \Big(\sum_{T\in {\cal T}_h} \|\bq\|_T^2\Big)^{\frac{1}{2}}\\&
 + \Big(\sum_{T\in {\cal T}_h} \| Q_0w-Q_bw\|_{\partial T} ^2\Big)^{\frac{1}{2}}\Big(\sum_{T\in {\cal T}_h}  \|  \bq \cdot\bn \|_{\partial T}^2\Big)^{\frac{1}{2}}\\
\leq &\Big(\sum_{T\in {\cal T}_h} \|
\nabla (w-Q_0w)\|^2_T\Big)^{\frac{1}{2}} \Big(\sum_{T\in {\cal T}_h} \|\bq\|_T^2\Big)^{\frac{1}{2}}\\&
 + \Big(\sum_{T\in {\cal T}_h} h_T^{-1}\| Q_0w-w\|_{T} ^2+h_T\| Q_0w-w\|_{1, T} ^2\Big)^{\frac{1}{2}}\Big(\sum_{T\in {\cal T}_h} h_T^{-1} \|\bq\|_{T}^2\Big)^{\frac{1}{2}}\\
&\leq Ch^{k}\|w\|_{k+1}\Big(\sum_{T\in {\cal T}_h} \|\bq\|_T^2\Big)^{\frac{1}{2}}.
\end{split}
\end{equation*}
Choosing $\bq=\nabla_{w}(w-Q_hw)$ yields
$$
\sum_{T\in {\cal T}_h} (\nabla_{w}(w-Q_hw), \nabla_{w}(w-Q_hw))_T\leq 
 Ch^{k}\|w\|_{k+1}\3bar w-Q_hw \3bar_{W_h},
 $$  
 which proves the claim.
\end{proof}

\begin{lemma}\label{theta}
Let $\btheta$ be the exact solution of the RM plate model \eqref{m1}, and suppose $\btheta \in [H^{k+1}(\Omega)]^2$. Then there exists a constant $C>0$ such that
\begin{equation}\label{erroresti2}
\3bar \btheta-\bQ_h  \btheta \3bar_{\Theta_h} \leq Ch^{k}\|\btheta\|_{k+1}.
\end{equation}
\end{lemma}
\begin{proof} 
Using \eqref{2.3new}, the Cauchy-Schwarz inequality, and the trace inequalities \eqref{tracein}--\eqref{trace}, for any $\beta \in [P_{r_2}(T)]^{2\times 2}$ we have
\begin{equation*}
\begin{split}
&\sum_{T\in {\cal T}_h} (\epsilon_{w}( \btheta-\bQ_h\btheta),  \beta)_T\\
 = &\sum_{T\in {\cal T}_h} (\epsilon(\btheta-\bQ_0\btheta),  \beta)_T+
  \langle \bQ_0 \btheta-\bQ_b \btheta,  \beta \cdot\bn \rangle_{\partial T}\\
\leq &\Big(\sum_{T\in {\cal T}_h} \|\epsilon(\btheta-\bQ_0\btheta)\|^2_T\Big)^{\frac{1}{2}} \Big(\sum_{T\in {\cal T}_h} \|\beta\|_T^2\Big)^{\frac{1}{2}}\\&
 + \Big(\sum_{T\in {\cal T}_h} \|\bQ_0 \btheta-\bQ_b \btheta\|_{\partial T} ^2\Big)^{\frac{1}{2}}\Big(\sum_{T\in {\cal T}_h}  \|  \beta \cdot\bn \|_{\partial T}^2\Big)^{\frac{1}{2}}\\
\leq &\Big(\sum_{T\in {\cal T}_h} \|
\epsilon(\btheta-\bQ_0\btheta)\|^2_T\Big)^{\frac{1}{2}} \Big(\sum_{T\in {\cal T}_h} \|\beta\|_T^2\Big)^{\frac{1}{2}}\\&
 + \Big(\sum_{T\in {\cal T}_h} h_T^{-1}\|\bQ_0 \btheta-  \btheta\|_{T} ^2+h_T\| \bQ_0 \btheta- \btheta\|_{1, T} ^2\Big)^{\frac{1}{2}}\Big(\sum_{T\in {\cal T}_h} h_T^{-1} \|\beta\|_{T}^2\Big)^{\frac{1}{2}}\\
&\leq Ch^{k}\|\btheta\|_{k+1}\Big(\sum_{T\in {\cal T}_h} \|\beta\|_T^2\Big)^{\frac{1}{2}}.
\end{split}
\end{equation*}
Letting $\beta=\epsilon_{w}( \btheta-\bQ_h\btheta)$ gives 
$$
\sum_{T\in {\cal T}_h} (\epsilon_{w}( \btheta-\bQ_h\btheta), \epsilon_{w}( \btheta-\bQ_h\btheta))_T\leq 
 Ch^{k}\|\btheta\|_{k+1}\3bar \btheta-\bQ_h\btheta \3bar_{\Theta_h},
 $$  
which completes the proof.
  
\end{proof}

  \begin{lemma}\cite{
  RM, mu}\label{beta}
     For any $\beta_h=\{\beta_0, \beta_b\}\in \Theta_h$, there holds  
 \begin{equation*} 
      (\sum_{T\in {\cal T}_h}\|\beta_0\|_T^2)^{\frac{1}{2}}\leq C\| \beta\|_{1, h}.
  \end{equation*}  
    
  \end{lemma}
\begin{theorem}

Let $(w,\btheta)$ be the exact solution of the RM plate model \eqref{m1}, and assume sufficient regularity such that $w \in H^{k+1}(\Omega)$ and $\btheta \in [H^{k+1}(\Omega)]^2$.
Let $(w_h, \btheta_h) \in W_h \times \Theta_h$ denote the numerical solution obtained from the WG scheme \ref{PDWG1}.
Then, there exists a constant $C>0$, independent of the mesh size $h$, such that the following error estimate holds:
\begin{equation}\label{trinorm}
 \3bar e_{\btheta_h}\3bar_{\Theta_h} +\lambda^{-\frac{1}{2}} t \|e_{\bgamma_h}\| + \3bar w-w_h\3bar_{W_h} \leq Ch^k(\|\bgamma\|_k+\|\btheta\|_{k+1}+\|w\|_{k+1}).
\end{equation}
\end{theorem}
\begin{proof}
From \eqref{pro} and \eqref{pro3}, we have
\begin{equation}\label{s1}
\begin{split}
    &\lambda^{-1}t^2( {\cal Q}_{r_1}\bgamma-\bgamma_h)\\
    =& ( {\cal Q}_{r_1}(\nabla w-\btheta)-(\nabla_w w_h-{\cal Q}_{r_1}\btheta_0)
    \\
    =& \nabla_w (Q_hw -w_h)-({\cal Q}_{r_1}\btheta-{\cal Q}_{r_1}\btheta_0)
\end{split}
\end{equation}

Letting $\beta= Q_h\btheta-\btheta_h$ and $v=Q_hw-w_h$ in \eqref{erroreqn}, we obtain
\begin{equation*} 
    \begin{split}
    &\sum_{T\in {\cal T}_h} ({\cal C} \epsilon_w (e_{\btheta_h}), \epsilon_w (Q_h\btheta-\btheta_h))_T  +(e_{\bgamma_h}, \nabla_w (Q_hw-w_h)- {\cal Q}_{r_1} (Q_0\btheta-\btheta_0))_T\\
    &-s(w_h, Q_hw-w_h)\\
    =&\sum_{T\in {\cal T}_h} ({\cal C} \epsilon_w (e_{\btheta_h}), \epsilon_w (e_{\btheta_h}))_T + ({\cal C} \epsilon_w (e_{\btheta_h}), \epsilon_w (Q_h\btheta-\btheta))_T \\
    &+(e_{\bgamma_h}, \lambda^{-1}t^2( {\cal Q}_{r_1}\bgamma-\bgamma_h))_T + (e_{\bgamma_h},   -{\cal Q}_{r_1} (Q_0\btheta-\btheta))_T  -s(w_h, Q_hw-w_h)\\
    =&\sum_{T\in {\cal T}_h} ({\cal C} \epsilon_w (e_{\btheta_h}), \epsilon_w (e_{\btheta_h}))_T + ({\cal C} \epsilon_w (e_{\btheta_h}), \epsilon_w (Q_h\btheta-\btheta))_T \\
    &+\lambda^{-1}t^2 (e_{\bgamma_h},  e_{\bgamma_h})_T +\lambda^{-1}t^2 (e_{\bgamma_h},     {\cal Q}_{r_1}\bgamma-\bgamma)_T+ (e_{\bgamma_h},   -{\cal Q}_{r_1} (Q_0\btheta-\btheta))_T\\ &-s(w_h, Q_hw-w_h)\\
    =&  \sum_{T\in {\cal T}_h}    {\cal C}(\epsilon (\btheta)-   {\cal Q}_{r_2} (\epsilon  \btheta))\cdot\bn, \beta_0- \beta_b\rangle_{\partial T}\\
    &+\langle v_0-v_b, (\bgamma -{\cal Q}_{r_1} \bgamma) \cdot \bn\rangle_{\partial T}  -s(w_h, Q_hw-w_h).
    \end{split}
\end{equation*} 
This can be written as 
\begin{equation} \label{sum}
    \begin{split}
    &\sum_{T\in {\cal T}_h} ({\cal C} \epsilon_w (e_{\btheta_h}), \epsilon_w (e_{\btheta_h}))_T  +\lambda^{-1}t^2 (e_{\bgamma_h},  e_{\bgamma_h})_T\\
    =& \sum_{T\in {\cal T}_h}  -  ({\cal C} \epsilon_w (e_{\btheta_h}), \epsilon_w (Q_h\btheta-\btheta))_T -\lambda^{-1}t^2 (e_{\bgamma_h},     {\cal Q}_{r_1}\bgamma-\bgamma)_T\\&+ (e_{\bgamma_h},    {\cal Q}_{r_1} (Q_0\btheta-\btheta))_T+ {\cal C}(\epsilon (\btheta)-   {\cal Q}_{r_2} (\epsilon  \btheta))\cdot\bn, \beta_0- \beta_b\rangle_{\partial T} \\
    &+\langle v_0-v_b, (\bgamma -{\cal Q}_{r_1} \bgamma) \cdot \bn\rangle_{\partial T}  -s(w_h, Q_hw-w_h) \\
     =& \sum_{i=1}^6 I_i.
    \end{split}
\end{equation} 
 We will estimate $I_i$ for $i=1, \cdots, 5$ term by term.  
 
 Estimate of $I_1$:  By Cauchy–Schwarz inequality, 
 \begin{equation*}
     \begin{split}
     |I_1|  \leq (\sum_{T\in {\cal T}_h} \|{\cal C} \epsilon_w (e_{\btheta_h})\|_T^2)^{\frac{1}{2}} (\sum_{T\in {\cal T}_h}\|\epsilon_w (Q_h\btheta-\btheta)\|_T  ^2)^{\frac{1}{2}} \leq \3bar e_{\btheta_h}\3bar_{\Theta_h} \3barQ_h\btheta-\btheta\3bar_{\Theta_h}. 
     \end{split}
 \end{equation*}

  Estimate of $I_2$: By Cauchy–Schwarz inequality, 
 \begin{equation*}
     \begin{split}
     |I_2|  \leq \lambda^{-1} t^2 (\sum_{T\in {\cal T}_h} \| e_{\bgamma_h}\|_T^2)^{\frac{1}{2}} (\sum_{T\in {\cal T}_h}\| {\cal Q}_{r_1}\bgamma-\bgamma \|_T  ^2)^{\frac{1}{2}} \leq Ch^k\|\bgamma\|_k \|e_{\bgamma_h}\| . 
     \end{split}
 \end{equation*}

 Estimate of $I_3$: Using Cauchy-Schwarz inequality, 
 \begin{equation*}
     \begin{split}
|I_3|\leq  &  (\sum_{T\in {\cal T}_h}   \|e_{\bgamma_h}\|_T^2)^{\frac{1}{2}}    (\sum_{T\in {\cal T}_h}  \|{\cal Q}_{r_1} (Q_0\btheta-\btheta)\|_T)^2)^{\frac{1}{2}}\\
\leq & Ch^{k+1}{\|\btheta\|_{k+1}} 
\|e_{\bgamma_h}\|.
   \end{split}
 \end{equation*}

 Estimate of $I_4$: Using Cauchy-Schwarz inequality, the trace inequality \eqref{tracein},  the triangle inequality, and  \eqref{normeq2}, we have
 \begin{equation*}
     \begin{split}
     |I_4|  \leq & (\sum_{T\in {\cal T}_h} \| \epsilon (\btheta)-   {\cal Q}_{r_2} (\epsilon  \btheta))\cdot\bn\|_{\partial T}^2)^{\frac{1}{2}} (\sum_{T\in {\cal T}_h}\| \beta_0- \beta_b \|_{\partial T}  ^2)^{\frac{1}{2}} \\ \leq &   (\sum_{T\in {\cal T}_h}  \| \epsilon (\btheta)-   {\cal Q}_{r_2} (\epsilon  \btheta))\cdot\bn\|_{  T}^2+h_T^2 \| \epsilon (\btheta)-   {\cal Q}_{r_2} (\epsilon  \btheta))\cdot\bn\|_{1,   T}^2)^{\frac{1}{2}} \\& \cdot  (\sum_{T\in {\cal T}_h}h_T^{-1}\| \beta_0- \beta_b \|_{\partial T}  ^2)^{\frac{1}{2}} \\
     \leq & Ch^k\|\btheta\|_{k+1} \|\3barQ_h\btheta-\btheta_h\|_{1, h}\\
     \leq & Ch^k\|\btheta\|_{k+1} (\3barQ_h\btheta-\btheta \3bar_{\Theta_h}+\3bar \btheta-\btheta_h\3bar_{\Theta_h}).
     \end{split}
 \end{equation*}

Estimate of $I_5$: Using Cauchy-Schwarz inequality, the trace inequality \eqref{tracein},  the triangle inequality, and \eqref{normeq}, we have
 \begin{equation*}
     \begin{split}
     |I_5|  \leq & (\sum_{T\in {\cal T}_h} \|v_0-v_b\|_{\partial T}^2)^{\frac{1}{2}} (\sum_{T\in {\cal T}_h}\|(\bgamma -{\cal Q}_{r_1} \bgamma) \cdot \bn\|_{\partial T}  ^2)^{\frac{1}{2}} \\ \leq &   (\sum_{T\in {\cal T}_h} h^{-1}\|v_0-v_b\|_{\partial T}^2)^{\frac{1}{2}} (\sum_{T\in {\cal T}_h}\|(\bgamma -{\cal Q}_{r_1} \bgamma) \cdot \bn\|_{ T}  ^2+h_T^2\|(\bgamma -{\cal Q}_{r_1} \bgamma) \cdot \bn\|_{1, T}  ^2)^{\frac{1}{2}}\\
     \leq & Ch^k\|\bgamma\|_{k} \|Q_hw-w_h\|_{1, h}\leq Ch^k\|\bgamma\|_{k} (\3bar Q_hw-w \3bar_{W_h}+\3bar w-w_h\3bar_{W_h}).
     \end{split}
 \end{equation*}

Estimate of $I_6$: Using    the triangle inequality, and \eqref{normeq}, we have
 \begin{equation*}
     \begin{split}
   |I_6|&\leq |s(w_h-w, Q_hw-w_h)|\\
   &\leq \|w_h-w\|^2_{1,h}+\|Q_hw-w_h\|^2_{1,h}\\
    &\leq \3bar w_h-w\3bar_{W_h} ^2 +\3bar Q_hw-w_h\3bar_{W_h} ^2.
       \end{split}
 \end{equation*}
Substituting the bounds for $I_i$ ($i=1,\dots,5$) into \eqref{sum} and applying Lemmas \ref{w} and \ref{theta}, we obtain
\begin{equation}\label{con1}
    \3bar e_{\btheta_h}\3bar_{\Theta_h} +\lambda^{-\frac{1}{2}} t \| e_{\bgamma_h}\| \leq Ch^k(\|\bgamma\|_k+\|\btheta\|_{k+1}+\|w\|_{k+1})+\3bar w_h-w\3bar_{W_h} ^2.
\end{equation}

Bound for $\3bar w - w_h \3bar_{W_h}$: Using \eqref{s1}, Lemma \ref{w}, Lemma \ref{beta},  \eqref{con1},  \eqref{normeq2} and the triangle inequality, we have
 \begin{equation*}
     \begin{split}
  & \3bar w-w_h\3bar _{W_h}^2 \\
    \leq & \sum_{T\in {\cal T}_h} \| \nabla_w (w-Q_hw)\|_T^2 +\| \nabla_w (Q_hw-w_h)\|_T^2 \\
    \leq  &   \3bar w-Q_hw\3bar_{W_h} ^2 +  \sum_{T\in {\cal T}_h}\|\lambda^{-1} t^2 ({\cal Q}_{r_1}\bgamma-\bgamma_h)\|_T^2+\|  {\cal Q}_{r_1} (\btheta-\btheta_0)\|_T^2\\
   \leq  & \3bar  w-Q_hw \3bar_{W_h}^2 +\sum_{T\in {\cal T}_h} \|\lambda^{-1} t^2 ({\cal Q}_{r_1}\bgamma-\bgamma)\|_T^2+\|\lambda^{-1} t^2  e_{\bgamma_h} \|_T^2+\|  \btheta-\btheta_h \|_{1,h} \\ 
   \leq  & \3bar  w-Q_hw \3bar_{W_h}^2+ Ch^{2k}\|\bgamma\|^2_k+\lambda^{-1} t^2\|  e_{\bgamma_h}\|^2 +\3bar e_{\btheta_h} \3bar_{\Theta_h}^2\\ 
   \leq & Ch^{2k}(\|\bgamma\|^2_k+\|\btheta\|^2_{k+1}+\|w\|^2_{k+1}).
     \end{split}
 \end{equation*}
Combining this estimate with \eqref{con1} completes the proof of the theorem.
\end{proof}

\section{Numerical experiments}
We compute the WG solutions for two problems, by the $P_1$, $P_2$ and $P_3$ weak Galerkin finite elements, 
  on three types of meshes.
  
For both examples,   the tensor of bending muduli in \eqref{m1} is defined by
\begin{equation}
\label{c-1}
    \mathcal{C}\btheta=\frac E{12(1-\nu^2)} [ (1-\nu) \btheta + \nu \operatorname{tr}(\btheta) I ],
\end{equation}
with Young's modulus $E=1.092$ and Poisson's ratio $\nu=0.3$. The shear correction factor is chosen so that in  \eqref{m1}  
\begin{equation}
\label{c-2}
    \lambda=\frac {E\kappa} {2(1+\nu)}, \quad\text{and \ } \kappa=\frac 56.
\end{equation}

In the first example, the domain is $(0,1)\times(0,1)$ and the body load is
\a{ g = \frac E{12(1-\nu^2)}   \Big(&12 y (y-1) (5 x^2-5 x+1)  (2 y^2
                     (y-1)^2\\
                +x (x-1) (5 y^2-5 y+1))  
                 +&12 x (x-1) (5 y^2-5 y+1)  (2 x^2 (x-1)^2\\
                 +y (y-1)  (5 x^2-5 x+1))\Big),& }
                  with constants defined in \eqref{c-1} and \eqref{c-2}.
The exact solution of \eqref{m1} is
\begin{equation}\label{s-1}\left\{\begin{aligned}
  \btheta(x,y)&= \begin{pmatrix}   y^3 (y-1)^3 x^2 (x-1)^2 (2 x-1) \\
                x^3 (x-1)^3 y^2 (y-1)^2 (2 y-1)  \end{pmatrix}, \\
 w(x,y) &=   \frac 13 x^3 (x-1)^3 y^3 (y-1)^3- \frac{2 t^2}
                  {5  (1-\nu)} \\
                   &\quad \  \big(y^3 (y-1)^3 x (x-1) (5 x^2-5 x+1) \\
         &\quad \quad                      +x^3 (x-1)^3 y (y-1) (5 y^2-5 y+1)\big).
       \end{aligned} \right.
      \end{equation}
   
In the second example, the body load in \eqref{m1} $g=1$.
The exact solution of \eqref{m1} with coefficients \eqref{c-1}--\eqref{c-2} is
\begin{equation}\label{s-2}\left\{\begin{aligned}
  \btheta(x,y)&=\frac 1{16D} \begin{pmatrix} x(x^2+y^2-1)\\ y(x^2+y^2-1) \end{pmatrix}, \\
 w(x,y) &=\frac{(x^2+y^2)^2}{64D}-(x^2+y^2-1)(\frac{t^2}{4\lambda}+\frac 1{32D})
           - \frac 1{64D}, \end{aligned} \right.
      \end{equation}
where $D=E/[12(1-\nu^2)]$.

\begin{figure}[H]
 \begin{center}\setlength\unitlength{1.0pt}
\begin{picture}(320,108)(0,0)
  \put(15,101){$G_1$:} \put(125,101){$G_2$:} \put(235,101){$G_3$:} 
  \put(0,-160){\includegraphics[width=380pt]{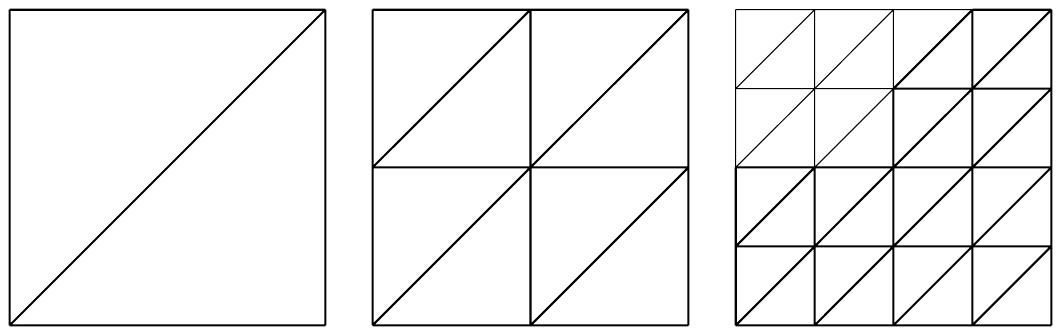}}  
 \end{picture}\end{center}
\caption{The triangular meshes for Tables \ref{t1}--\ref{t6}. }\label{f-1}
\end{figure}

We denote a finite element pair $W_h^0/\Theta_h^0$ by  $P_{k}$-$P_{p}$-$P_{r_1}$/$P_{q}$-$P_{t}$-$P_{r_2}$ for the $k$ and $p$ in \eqref{W-T},
   the $r_1$ in \eqref{2.4}, the $q$ and $t$ in \eqref{t-T}, and the $r_2$ in  \eqref{2.3}, in the tables of computational results.

 In Tables \ref{t1}--\ref{t3},  we list the results for three finite elements when computing the solution \eqref{s-1} on
   triangular meshes.  We can see the optimal order of convergence is achieved in all cases.
   There is no locking for thin plate $t=10^{-2}$.
In particular, when $t=10^{-2}$, we have one order of super-convergence for $w_h$, for all three finite elements.

  \begin{table}[H]
  \caption{ Error profile for computing \eqref{s-1} by the $P_{1}$-$P_{1}$-$P_{1}$/$P_{1}$-$P_{1}$-$P_{2}$ element on triangular
        meshes shown in Figure \ref{f-1}.} \label{t1}
\begin{center}  
   \begin{tabular}{c|rr|rr}  
 \hline 
&\multicolumn{4}{c}{ $t=1$ in \eqref{m1}.  }\\ \hline
$G_i$ &  $ \|Q_h w - w_h \| $ & $O(h^r)$ &  $ \3bar Q_h w- w_h\3bar_{W_h} $ & $O(h^r)$   \\  
 4&    0.183E-04 &  1.7&    0.422E-03 &  0.9\\
 5&    0.492E-05 &  1.9&    0.212E-03 &  1.0\\
 6&    0.125E-05 &  2.0&    0.106E-03 &  1.0\\
 \hline 
 &  $ \|Q_h \btheta -\btheta_h \| $ & $O(h^r)$ &  $ \3bar Q_h \btheta- \btheta_h\3bar_{\Theta_h} $ & $O(h^r)$   \\ 
 4&    0.887E-05 &  1.7&    0.572E-04 &  1.1\\
 5&    0.240E-05 &  1.9&    0.271E-04 &  1.1\\
 6&    0.616E-06 &  2.0&    0.131E-04 &  1.0\\
 \hline 
&\multicolumn{4}{c}{ $t=10^{-2}$ in \eqref{m1}.  }\\
 \hline  
$G_i$ &  $ \|Q_h w - w_h \| $ & $O(h^r)$ &  $ \3bar Q_h w- w_h\3bar_{W_h} $ & $O(h^r)$   \\ 
 4&    0.417E-05 &  2.9&    0.339E-05 &  1.9\\
 5&    0.541E-06 &  2.9&    0.939E-06 &  1.9\\
 6&    0.713E-07 &  2.9&    0.245E-06 &  1.9\\
 \hline 
 &  $ \|Q_h \btheta -\btheta_h \| $ & $O(h^r)$ &  $ \3bar Q_h \btheta- \btheta_h\3bar_{\Theta_h} $ & $O(h^r)$   \\  
 4&    0.338E-05 &  1.8&    0.526E-04 &  0.9\\
 5&    0.939E-06 &  1.8&    0.263E-04 &  1.0\\
 6&    0.245E-06 &  1.9&    0.130E-04 &  1.0\\
\hline 
\end{tabular} \end{center}  \end{table}
 
  \begin{table}[H]
  \caption{ Error profile for computing \eqref{s-1} by the $P_{2}$-$P_{2}$-$P_{2}$/$P_{2}$-$P_{2}$-$P_{3}$ element on triangular
        meshes shown in Figure \ref{f-1}.} \label{t2}
\begin{center}  
   \begin{tabular}{c|rr|rr}  
 \hline 
&\multicolumn{4}{c}{ $t=1$ in \eqref{m1}.  }\\ \hline
$G_i$ &  $ \|Q_h w - w_h \| $ & $O(h^r)$ &  $ \3bar Q_h w- w_h\3bar_{W_h} $ & $O(h^r)$   \\ 
 3&    0.791E-05 &  2.9&    0.204E-03 &  1.9\\
 4&    0.886E-06 &  3.2&    0.523E-04 &  2.0\\
 5&    0.104E-06 &  3.1&    0.131E-04 &  2.0\\
 \hline 
 &  $ \|Q_h \btheta -\btheta_h \| $ & $O(h^r)$ &  $ \3bar Q_h \btheta- \btheta_h\3bar_{\Theta_h} $ & $O(h^r)$   \\  
 3&    0.366E-05 &  2.7&    0.368E-04 &  1.7\\
 4&    0.414E-06 &  3.1&    0.102E-04 &  1.9\\
 5&    0.467E-07 &  3.1&    0.262E-05 &  2.0\\
 \hline 
&\multicolumn{4}{c}{ $t=10^{-2}$ in \eqref{m1}.  }\\
 \hline  
$G_i$ &  $ \|Q_h w - w_h \| $ & $O(h^r)$ &  $ \3bar Q_h w- w_h\3bar_{W_h} $ & $O(h^r)$   \\ 
 3&    0.621E-05 &  3.7&    0.254E-05 &  3.1\\
 4&    0.429E-06 &  3.9&    0.268E-06 &  3.2\\
 5&    0.277E-07 &  4.0&    0.302E-07 &  3.1\\
 \hline 
 &  $ \|Q_h \btheta -\btheta_h \| $ & $O(h^r)$ &  $ \3bar Q_h \btheta- \btheta_h\3bar_{\Theta_h} $ & $O(h^r)$   \\  
 3&    0.227E-05 &  2.9&    0.329E-04 &  1.7\\
 4&    0.256E-06 &  3.2&    0.926E-05 &  1.8\\
 5&    0.304E-07 &  3.1&    0.245E-05 &  1.9\\ 
\hline 
\end{tabular} \end{center}  \end{table}
 
  \begin{table}[H]
  \caption{ Error profile for computing \eqref{s-1} by the $P_{3}$-$P_{3}$-$P_{3}$/$P_{3}$-$P_{3}$-$P_{4}$ element on triangular
        meshes shown in Figure \ref{f-1}.} \label{t3}
\begin{center}  
   \begin{tabular}{c|rr|rr}  
 \hline 
&\multicolumn{4}{c}{ $t=1$ in \eqref{m1}.  }\\ \hline
$G_i$ &  $ \|Q_h w - w_h \| $ & $O(h^r)$ &  $ \3bar Q_h w- w_h\3bar_{W_h} $ & $O(h^r)$   \\ 
 4&    0.646E-07 &  4.2&    0.534E-05 &  3.0\\
 5&    0.365E-08 &  4.1&    0.680E-06 &  3.0\\
 6&    0.218E-09 &  4.1&    0.856E-07 &  3.0\\
 \hline 
 &  $ \|Q_h \btheta -\btheta_h \| $ & $O(h^r)$ &  $ \3bar Q_h \btheta- \btheta_h\3bar_{\Theta_h} $ & $O(h^r)$   \\  
 4&    0.395E-07 &  4.0&    0.124E-05 &  2.8\\
 5&    0.236E-08 &  4.1&    0.159E-06 &  3.0\\
 6&    0.143E-09 &  4.0&    0.200E-07 &  3.0\\
 \hline 
&\multicolumn{4}{c}{ $t=10^{-2}$ in \eqref{m1}.  }\\
 \hline  
$G_i$ &  $ \|Q_h w - w_h \| $ & $O(h^r)$ &  $ \3bar Q_h w- w_h\3bar_{W_h} $ & $O(h^r)$   \\ 
 4&    0.396E-07 &  4.7&    0.306E-07 &  4.1\\
 5&    0.133E-08 &  4.9&    0.178E-08 &  4.1\\
 6&    0.427E-10 &  5.0&    0.110E-09 &  4.0\\
 \hline 
 &  $ \|Q_h \btheta -\btheta_h \| $ & $O(h^r)$ &  $ \3bar Q_h \btheta- \btheta_h\3bar_{\Theta_h} $ & $O(h^r)$   \\  
 4&    0.281E-07 &  3.9&    0.114E-05 &  2.8\\
 5&    0.178E-08 &  4.0&    0.150E-06 &  2.9\\
 6&    0.119E-09 &  3.9&    0.194E-07 &  3.0\\
\hline 
\end{tabular} \end{center}  \end{table}

  \begin{table}[H]
  \caption{ Error profile for computing \eqref{s-2} by the $P_{1}$-$P_{1}$-$P_{1}$/$P_{1}$-$P_{1}$-$P_{2}$ element on triangular
        meshes shown in Figure \ref{f-1}.} \label{t4}
\begin{center}  
   \begin{tabular}{c|rr|rr}  
 \hline 
&\multicolumn{4}{c}{ $t=1$ in \eqref{m1}.  }\\ \hline
$G_i$ &  $ \|Q_h w - w_h \| $ & $O(h^r)$ &  $ \3bar Q_h w- w_h\3bar_{W_h} $ & $O(h^r)$   \\  
 4&    0.443E-03 &  1.9&    0.859E-02 &  1.0\\
 5&    0.113E-03 &  2.0&    0.427E-02 &  1.0\\
 6&    0.285E-04 &  2.0&    0.213E-02 &  1.0\\
 \hline 
 &  $ \|Q_h \btheta -\btheta_h \| $ & $O(h^r)$ &  $ \3bar Q_h \btheta- \btheta_h\3bar_{\Theta_h} $ & $O(h^r)$   \\ 
 4&    0.254E-03 &  1.8&    0.216E-02 &  1.0\\
 5&    0.662E-04 &  1.9&    0.107E-02 &  1.0\\
 6&    0.168E-04 &  2.0&    0.534E-03 &  1.0\\
 \hline 
&\multicolumn{4}{c}{ $t=10^{-2}$ in \eqref{m1}.  }\\
 \hline  
$G_i$ &  $ \|Q_h w - w_h \| $ & $O(h^r)$ &  $ \3bar Q_h w- w_h\3bar_{W_h} $ & $O(h^r)$   \\ 
 4&    0.627E-04 &  2.8&    0.154E-03 &  1.8\\
 5&    0.105E-04 &  2.6&    0.400E-04 &  1.9\\
 6&    0.216E-05 &  2.3&    0.101E-04 &  2.0\\
 \hline 
 &  $ \|Q_h \btheta -\btheta_h \| $ & $O(h^r)$ &  $ \3bar Q_h \btheta- \btheta_h\3bar_{\Theta_h} $ & $O(h^r)$   \\  
 4&    0.154E-03 &  1.8&    0.207E-02 &  0.9\\
 5&    0.400E-04 &  1.9&    0.106E-02 &  1.0\\
 6&    0.101E-04 &  2.0&    0.531E-03 &  1.0\\
\hline 
\end{tabular} \end{center}  \end{table}
 
 In Tables \ref{t4}--\ref{t6},  we list the results for three finite elements when computing the solution \eqref{s-2} on
   triangular meshes.  We can see the optimal order of convergence is achieved in all cases.
   There is no locking for thin plate $t=10^{-2}$.
   But unlike the first case of \eqref{s-1}, as the error of $\btheta_h$ dominates,  we have
     a deteriorate super-convergence for $w_h$ in the $L^2$ norm.
It seems we still have one order of super-convergence for $w_h$ in the triple-bar norm.

  \begin{table}[H]
  \caption{ Error profile for computing \eqref{s-2} by the $P_{2}$-$P_{2}$-$P_{2}$/$P_{2}$-$P_{2}$-$P_{3}$ element on triangular
        meshes shown in Figure \ref{f-1}.} \label{t5}
\begin{center}  
   \begin{tabular}{c|rr|rr}  
 \hline 
&\multicolumn{4}{c}{ $t=1$ in \eqref{m1}.  }\\ \hline
$G_i$ &  $ \|Q_h w - w_h \| $ & $O(h^r)$ &  $ \3bar Q_h w- w_h\3bar_{W_h} $ & $O(h^r)$   \\ 
 4&    0.116E-04 &  3.2&    0.776E-03 &  1.9\\
 5&    0.138E-05 &  3.1&    0.196E-03 &  2.0\\
 6&    0.168E-06 &  3.0&    0.492E-04 &  2.0\\
 \hline 
 &  $ \|Q_h \btheta -\btheta_h \| $ & $O(h^r)$ &  $ \3bar Q_h \btheta- \btheta_h\3bar_{\Theta_h} $ & $O(h^r)$   \\  
 4&    0.987E-05 &  3.1&    0.246E-03 &  1.9\\
 5&    0.116E-05 &  3.1&    0.624E-04 &  2.0\\
 6&    0.142E-06 &  3.0&    0.157E-04 &  2.0\\
 \hline 
&\multicolumn{4}{c}{ $t=10^{-2}$ in \eqref{m1}.  }\\
 \hline  
$G_i$ &  $ \|Q_h w - w_h \| $ & $O(h^r)$ &  $ \3bar Q_h w- w_h\3bar_{W_h} $ & $O(h^r)$   \\ 
 4&    0.550E-05 &  3.9&    0.677E-05 &  3.1\\
 5&    0.353E-06 &  4.0&    0.811E-06 &  3.1\\
 6&    0.232E-07 &  3.9&    0.103E-06 &  3.0\\
 \hline 
 &  $ \|Q_h \btheta -\btheta_h \| $ & $O(h^r)$ &  $ \3bar Q_h \btheta- \btheta_h\3bar_{\Theta_h} $ & $O(h^r)$   \\  
 4&    0.660E-05 &  3.1&    0.230E-03 &  1.9\\
 5&    0.815E-06 &  3.0&    0.591E-04 &  2.0\\
 6&    0.110E-06 &  2.9&    0.151E-04 &  2.0\\
\hline 
\end{tabular} \end{center}  \end{table}
 
  \begin{table}[H]
  \caption{ Error profile for computing \eqref{s-2} by the $P_{3}$-$P_{3}$-$P_{3}$/$P_{3}$-$P_{3}$-$P_{4}$ element on triangular
        meshes shown in Figure \ref{f-1}.} \label{t6}
\begin{center}  
   \begin{tabular}{c|rr|rr}  
 \hline 
&\multicolumn{4}{c}{ $t=1$ in \eqref{m1}.  }\\ \hline
$G_i$ &  $ \|Q_h w - w_h \| $ & $O(h^r)$ &  $ \3bar Q_h w- w_h\3bar_{W_h} $ & $O(h^r)$   \\ 
 4&    0.741E-06 &  4.1&    0.667E-04 &  2.9\\
 5&    0.442E-07 &  4.1&    0.845E-05 &  3.0\\
 6&    0.270E-08 &  4.0&    0.106E-05 &  3.0\\
 \hline 
 &  $ \|Q_h \btheta -\btheta_h \| $ & $O(h^r)$ &  $ \3bar Q_h \btheta- \btheta_h\3bar_{\Theta_h} $ & $O(h^r)$   \\  
 4&    0.664E-06 &  4.0&    0.224E-04 &  2.8\\
 5&    0.408E-07 &  4.0&    0.288E-05 &  3.0\\
 6&    0.251E-08 &  4.0&    0.364E-06 &  3.0\\
 \hline 
&\multicolumn{4}{c}{ $t=10^{-2}$ in \eqref{m1}.  }\\
 \hline  
$G_i$ &  $ \|Q_h w - w_h \| $ & $O(h^r)$ &  $ \3bar Q_h w- w_h\3bar_{W_h} $ & $O(h^r)$   \\ 
 4&    0.292E-06 &  5.0&    0.523E-06 &  4.1\\
 5&    0.929E-08 &  5.0&    0.323E-07 &  4.0\\
 6&    0.302E-09 &  4.9&    0.205E-08 &  4.0\\
 \hline 
 &  $ \|Q_h \btheta -\btheta_h \| $ & $O(h^r)$ &  $ \3bar Q_h \btheta- \btheta_h\3bar_{\Theta_h} $ & $O(h^r)$   \\  
 4&    0.501E-06 &  3.9&    0.210E-04 &  2.8\\
 5&    0.327E-07 &  3.9&    0.276E-05 &  2.9\\
 6&    0.218E-08 &  3.9&    0.356E-06 &  3.0\\
\hline 
\end{tabular} \end{center}  \end{table}

\begin{figure}[H]
 \begin{center}\setlength\unitlength{1.0pt}
\begin{picture}(320,108)(0,0)
  \put(15,101){$G_1$:} \put(125,101){$G_2$:} \put(235,101){$G_3$:} 
  \put(0,-160){\includegraphics[width=380pt]{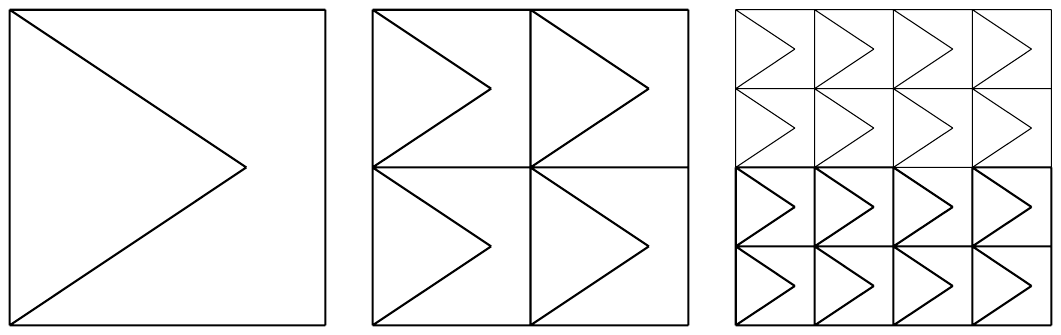}}  
 \end{picture}\end{center}
\caption{The non-convex polygonal meshes for Tables \ref{t7}--\ref{t12}. }\label{f-2}
\end{figure}

 In Tables \ref{t7}--\ref{t12}, the computation is done on non-convex polygonal meshes shown in Figure \ref{f-2}.
The optimal order of convergence is obtained in all cases.
In particular, when $t=10^{-2}$,  we still have one order of super-convergence for $w_h$, even on 
  the nonconvex polygonal meshes.

  \begin{table}[H]
  \caption{ Error profile for computing \eqref{s-1} by the $P_{1}$-$P_{1}$-$P_{1}$/$P_{1}$-$P_{1}$-$P_{3}$ element on polygonal
        meshes shown in Figure \ref{f-2}.} \label{t7}
\begin{center}  
   \begin{tabular}{c|rr|rr}  
 \hline 
&\multicolumn{4}{c}{ $t=1$ in \eqref{m1}.  }\\ \hline
$G_i$ &  $ \|Q_h w - w_h \| $ & $O(h^r)$ &  $ \3bar Q_h w- w_h\3bar_{W_h} $ & $O(h^r)$   \\  
 4&    0.231E-04 &  1.7&    0.488E-03 &  1.0\\
 5&    0.721E-05 &  1.7&    0.233E-03 &  1.1\\
 6&    0.197E-05 &  1.9&    0.113E-03 &  1.0\\
 \hline 
 &  $ \|Q_h \btheta -\btheta_h \| $ & $O(h^r)$ &  $ \3bar Q_h \btheta- \btheta_h\3bar_{\Theta_h} $ & $O(h^r)$   \\ 
 4&    0.211E-04 &  1.4&    0.142E-03 &  1.2\\
 5&    0.658E-05 &  1.7&    0.689E-04 &  1.0\\
 6&    0.178E-05 &  1.9&    0.338E-04 &  1.0\\
 \hline 
&\multicolumn{4}{c}{ $t=10^{-2}$ in \eqref{m1}.  }\\
 \hline  
$G_i$ &  $ \|Q_h w - w_h \| $ & $O(h^r)$ &  $ \3bar Q_h w- w_h\3bar_{W_h} $ & $O(h^r)$   \\ 
 4&    0.517E-05 &  2.8&    0.152E-04 &  1.5\\
 5&    0.877E-06 &  2.6&    0.444E-05 &  1.8\\
 6&    0.181E-06 &  2.3&    0.115E-05 &  1.9\\ 
 \hline 
 &  $ \|Q_h \btheta -\btheta_h \| $ & $O(h^r)$ &  $ \3bar Q_h \btheta- \btheta_h\3bar_{\Theta_h} $ & $O(h^r)$   \\  
 4&    0.150E-04 &  1.5&    0.131E-03 &  1.2\\
 5&    0.443E-05 &  1.8&    0.654E-04 &  1.0\\
 6&    0.115E-05 &  1.9&    0.327E-04 &  1.0\\
\hline 
\end{tabular} \end{center}  \end{table}

  \begin{table}[H]
  \caption{ Error profile for computing \eqref{s-1} by the $P_{2}$-$P_{2}$-$P_{2}$/$P_{2}$-$P_{2}$-$P_{4}$ element on polygonal
        meshes shown in Figure \ref{f-2}.} \label{t8}
\begin{center}  
   \begin{tabular}{c|rr|rr}  
 \hline 
&\multicolumn{4}{c}{ $t=1$ in \eqref{m1}.  }\\ \hline
$G_i$ &  $ \|Q_h w - w_h \| $ & $O(h^r)$ &  $ \3bar Q_h w- w_h\3bar_{W_h} $ & $O(h^r)$   \\  
 3&    0.267E-04 &  5.4&    0.667E-03 &  4.0\\
 4&    0.144E-05 &  4.2&    0.741E-04 &  3.2\\
 5&    0.158E-06 &  3.2&    0.169E-04 &  2.1\\
 \hline 
 &  $ \|Q_h \btheta -\btheta_h \| $ & $O(h^r)$ &  $ \3bar Q_h \btheta- \btheta_h\3bar_{\Theta_h} $ & $O(h^r)$   \\ 
 3&    0.123E-04 &  5.0&    0.196E-03 &  3.0\\
 4&    0.100E-05 &  3.6&    0.280E-04 &  2.8\\
 5&    0.118E-06 &  3.1&    0.712E-05 &  2.0\\
 \hline 
&\multicolumn{4}{c}{ $t=10^{-2}$ in \eqref{m1}.  }\\
 \hline  
$G_i$ &  $ \|Q_h w - w_h \| $ & $O(h^r)$ &  $ \3bar Q_h w- w_h\3bar_{W_h} $ & $O(h^r)$   \\ 
 3&    0.823E-05 &  5.5&    0.234E-04 &  4.1\\
 4&    0.521E-06 &  4.0&    0.166E-05 &  3.8\\
 5&    0.340E-07 &  3.9&    0.137E-06 &  3.6\\
 \hline 
 &  $ \|Q_h \btheta -\btheta_h \| $ & $O(h^r)$ &  $ \3bar Q_h \btheta- \btheta_h\3bar_{\Theta_h} $ & $O(h^r)$   \\  
 3&    0.105E-04 &  4.7&    0.197E-03 &  3.0\\
 4&    0.826E-06 &  3.7&    0.272E-04 &  2.9\\
 5&    0.103E-06 &  3.0&    0.687E-05 &  2.0\\
\hline 
\end{tabular} \end{center}  \end{table}

  \begin{table}[H]
  \caption{ Error profile for computing \eqref{s-1} by the $P_3$-$P_3$-$P_3$/$P_3$-$P_3$-$P_5$ element on polygonal
        meshes shown in Figure \ref{f-2}.} \label{t9}
\begin{center}  
   \begin{tabular}{c|rr|rr}  
 \hline 
&\multicolumn{4}{c}{ $t=1$ in \eqref{m1}.  }\\ \hline
$G_i$ &  $ \|Q_h w - w_h \| $ & $O(h^r)$ &  $ \3bar Q_h w- w_h\3bar_{W_h} $ & $O(h^r)$   \\  
 3&    0.170E-04 &  6.2&    0.840E-03 &  3.7\\
 4&    0.269E-06 &  6.0&    0.279E-04 &  4.9\\
 5&    0.745E-08 &  5.2&    0.117E-05 &  4.6\\
 \hline 
 &  $ \|Q_h \btheta -\btheta_h \| $ & $O(h^r)$ &  $ \3bar Q_h \btheta- \btheta_h\3bar_{\Theta_h} $ & $O(h^r)$   \\ 
 3&    0.489E-05 &  6.6&    0.120E-03 &  5.1\\
 4&    0.117E-06 &  5.4&    0.500E-05 &  4.6\\
 5&    0.678E-08 &  4.1&    0.482E-06 &  3.4\\
 \hline 
&\multicolumn{4}{c}{ $t=10^{-2}$ in \eqref{m1}.  }\\
 \hline  
$G_i$ &  $ \|Q_h w - w_h \| $ & $O(h^r)$ &  $ \3bar Q_h w- w_h\3bar_{W_h} $ & $O(h^r)$   \\ 
 3&    0.131E-05 &  6.4&    0.210E-04 &  4.3\\
 4&    0.488E-07 &  4.7&    0.718E-06 &  4.9\\
 5&    0.166E-08 &  4.9&    0.233E-07 &  4.9\\
 \hline 
 &  $ \|Q_h \btheta -\btheta_h \| $ & $O(h^r)$ &  $ \3bar Q_h \btheta- \btheta_h\3bar_{\Theta_h} $ & $O(h^r)$   \\  
 3&    0.428E-05 &  6.6&    0.120E-03 &  5.1\\
 4&    0.925E-07 &  5.5&    0.478E-05 &  4.7\\
 5&    0.468E-08 &  4.3&    0.438E-06 &  3.4\\
\hline 
\end{tabular} \end{center}  \end{table}

  \begin{table}[H]
  \caption{ Error profile for computing \eqref{s-2} by the $P_{1}$-$P_{1}$-$P_{1}$/$P_{1}$-$P_{1}$-$P_{3}$ element on polygonal
        meshes shown in Figure \ref{f-2}.} \label{t10}
\begin{center}  
   \begin{tabular}{c|rr|rr}  
 \hline 
&\multicolumn{4}{c}{ $t=1$ in \eqref{m1}.  }\\ \hline
$G_i$ &  $ \|Q_h w - w_h \| $ & $O(h^r)$ &  $ \3bar Q_h w- w_h\3bar_{W_h} $ & $O(h^r)$   \\  
 3&    0.868E-04 &  3.7&    0.109E-02 &  2.6\\
 4&    0.240E-04 &  1.9&    0.470E-03 &  1.2\\
 5&    0.697E-05 &  1.8&    0.237E-03 &  1.0\\
 \hline 
 &  $ \|Q_h \btheta -\btheta_h \| $ & $O(h^r)$ &  $ \3bar Q_h \btheta- \btheta_h\3bar_{\Theta_h} $ & $O(h^r)$   \\ 
 3&    0.596E-04 &  3.0&    0.601E-03 &  2.5\\
 4&    0.203E-04 &  1.6&    0.229E-03 &  1.4\\
 5&    0.584E-05 &  1.8&    0.113E-03 &  1.0\\
 \hline 
&\multicolumn{4}{c}{ $t=10^{-2}$ in \eqref{m1}.  }\\
 \hline  
$G_i$ &  $ \|Q_h w - w_h \| $ & $O(h^r)$ &  $ \3bar Q_h w- w_h\3bar_{W_h} $ & $O(h^r)$   \\ 
 3&    0.393E-04 &  4.0&    0.540E-04 &  2.6\\
 4&    0.505E-05 &  3.0&    0.150E-04 &  1.9\\
 5&    0.766E-06 &  2.7&    0.394E-05 &  1.9\\
 \hline 
 &  $ \|Q_h \btheta -\btheta_h \| $ & $O(h^r)$ &  $ \3bar Q_h \btheta- \btheta_h\3bar_{\Theta_h} $ & $O(h^r)$   \\  
 3&    0.470E-04 &  3.4&    0.584E-03 &  2.6\\
 4&    0.142E-04 &  1.7&    0.224E-03 &  1.4\\
 5&    0.388E-05 &  1.9&    0.112E-03 &  1.0\\
\hline 
\end{tabular} \end{center}  \end{table}

  \begin{table}[H]
  \caption{ Error profile for computing \eqref{s-2} by the $P_{2}$-$P_{2}$-$P_{2}$/$P_{2}$-$P_{2}$-$P_{4}$ element on polygonal
        meshes shown in Figure \ref{f-2}.} \label{t11}
\begin{center}  
   \begin{tabular}{c|rr|rr}  
 \hline 
&\multicolumn{4}{c}{ $t=1$ in \eqref{m1}.  }\\ \hline
$G_i$ &  $ \|Q_h w - w_h \| $ & $O(h^r)$ &  $ \3bar Q_h w- w_h\3bar_{W_h} $ & $O(h^r)$   \\  
 3&    0.355E-04 &  6.2&    0.103E-02 &  4.0\\
 4&    0.162E-05 &  4.5&    0.742E-04 &  3.8\\
 5&    0.162E-06 &  3.3&    0.158E-04 &  2.2\\
 \hline 
 &  $ \|Q_h \btheta -\btheta_h \| $ & $O(h^r)$ &  $ \3bar Q_h \btheta- \btheta_h\3bar_{\Theta_h} $ & $O(h^r)$   \\ 
 3&    0.133E-04 &  5.6&    0.295E-03 &  3.5\\
 4&    0.907E-06 &  3.9&    0.392E-04 &  2.9\\
 5&    0.106E-06 &  3.1&    0.953E-05 &  2.0\\
 \hline 
&\multicolumn{4}{c}{ $t=10^{-2}$ in \eqref{m1}.  }\\
 \hline  
$G_i$ &  $ \|Q_h w - w_h \| $ & $O(h^r)$ &  $ \3bar Q_h w- w_h\3bar_{W_h} $ & $O(h^r)$   \\ 
 3&    0.800E-05 &  5.8&    0.364E-04 &  4.2\\
 4&    0.484E-06 &  4.0&    0.238E-05 &  3.9\\
 5&    0.318E-07 &  3.9&    0.172E-06 &  3.8\\
 \hline 
 &  $ \|Q_h \btheta -\btheta_h \| $ & $O(h^r)$ &  $ \3bar Q_h \btheta- \btheta_h\3bar_{\Theta_h} $ & $O(h^r)$   \\  
 3&    0.129E-04 &  5.4&    0.295E-03 &  3.5\\
 4&    0.863E-06 &  3.9&    0.391E-04 &  2.9\\
 5&    0.101E-06 &  3.1&    0.946E-05 &  2.0\\
\hline 
\end{tabular} \end{center}  \end{table}

  \begin{table}[H]
  \caption{ Error profile for computing \eqref{s-2} by the $P_3$-$P_3$-$P_3$/$P_3$-$P_3$-$P_5$ element on polygonal
        meshes shown in Figure \ref{f-2}.} \label{t12}
\begin{center}  
   \begin{tabular}{c|rr|rr}  
 \hline 
&\multicolumn{4}{c}{ $t=1$ in \eqref{m1}.  }\\ \hline
$G_i$ &  $ \|Q_h w - w_h \| $ & $O(h^r)$ &  $ \3bar Q_h w- w_h\3bar_{W_h} $ & $O(h^r)$   \\  
 2&    0.242E-02 &  5.6&    0.211E-01 &  4.7\\
 3&    0.207E-04 &  6.9&    0.904E-03 &  4.5\\
 4&    0.302E-06 &  6.1&    0.287E-04 &  5.0\\
 \hline 
 &  $ \|Q_h \btheta -\btheta_h \| $ & $O(h^r)$ &  $ \3bar Q_h \btheta- \btheta_h\3bar_{\Theta_h} $ & $O(h^r)$   \\ 
 2&    0.688E-03 &  5.5&    0.429E-02 &  4.6\\
 3&    0.470E-05 &  7.2&    0.133E-03 &  5.0\\
 4&    0.931E-07 &  5.7&    0.575E-05 &  4.5\\
 \hline 
&\multicolumn{4}{c}{ $t=10^{-2}$ in \eqref{m1}.  }\\
 \hline  
$G_i$ &  $ \|Q_h w - w_h \| $ & $O(h^r)$ &  $ \3bar Q_h w- w_h\3bar_{W_h} $ & $O(h^r)$   \\ 
 2&    0.165E-03 &  6.5&    0.767E-03 &  5.1\\
 3&    0.133E-05 &  7.0&    0.228E-04 &  5.1\\
 4&    0.478E-07 &  4.8&    0.764E-06 &  4.9\\
 \hline 
 &  $ \|Q_h \btheta -\btheta_h \| $ & $O(h^r)$ &  $ \3bar Q_h \btheta- \btheta_h\3bar_{\Theta_h} $ & $O(h^r)$   \\  
 2&    0.623E-03 &  5.5&    0.426E-02 &  4.6\\
 3&    0.470E-05 &  7.0&    0.134E-03 &  5.0\\
 4&    0.946E-07 &  5.6&    0.576E-05 &  4.5\\
\hline 
\end{tabular} \end{center}  \end{table}

\begin{figure}[H]
 \begin{center}\setlength\unitlength{1.0pt}
\begin{picture}(320,108)(0,0)
  \put(15,101){$G_1$:} \put(125,101){$G_2$:} \put(235,101){$G_3$:} 
  \put(0,-160){\includegraphics[width=380pt]{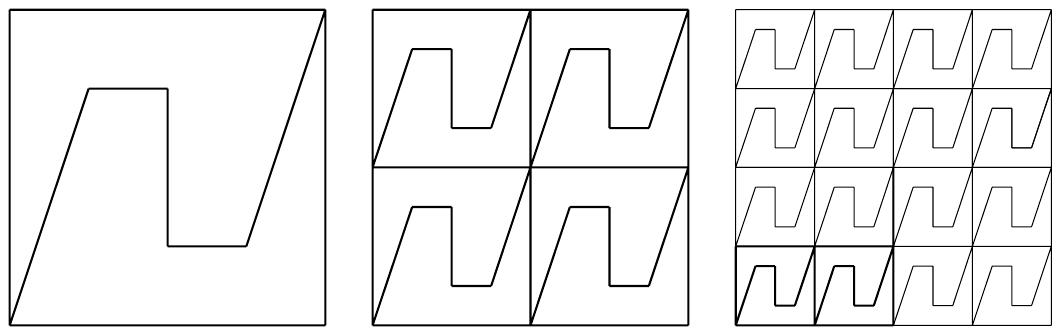}}  
 \end{picture}\end{center}
\caption{The non-convex polygonal meshes for Tables \ref{t13}--\ref{t18}. }\label{f-3}
\end{figure}

 In Tables \ref{t13}--\ref{t18}, the computation is done on non-convex polygonal meshes shown in Figure \ref{f-3}.
The optimal order of convergence is obtained in all cases.
And it shows the scheme is locking-free.

  \begin{table}[H]
  \caption{ Error profile for computing \eqref{s-1} by the $P_{1}$-$P_{1}$-$P_{1}$/$P_{1}$-$P_{1}$-$P_{4}$ element on polygonal
        meshes shown in Figure \ref{f-3}.} \label{t13}
\begin{center}  
   \begin{tabular}{c|rr|rr}  
 \hline 
&\multicolumn{4}{c}{ $t=1$ in \eqref{m1}.  }\\ \hline
$G_i$ &  $ \|Q_h w - w_h \| $ & $O(h^r)$ &  $ \3bar Q_h w- w_h\3bar_{W_h} $ & $O(h^r)$   \\  
 3&    0.868E-04 &  3.7&    0.109E-02 &  2.6\\
 4&    0.240E-04 &  1.9&    0.470E-03 &  1.2\\
 5&    0.697E-05 &  1.8&    0.237E-03 &  1.0\\ 
 \hline 
 &  $ \|Q_h \btheta -\btheta_h \| $ & $O(h^r)$ &  $ \3bar Q_h \btheta- \btheta_h\3bar_{\Theta_h} $ & $O(h^r)$   \\ 
 3&    0.596E-04 &  3.0&    0.601E-03 &  2.5\\
 4&    0.203E-04 &  1.6&    0.229E-03 &  1.4\\
 5&    0.584E-05 &  1.8&    0.113E-03 &  1.0\\ 
 \hline 
&\multicolumn{4}{c}{ $t=10^{-2}$ in \eqref{m1}.  }\\
 \hline  
$G_i$ &  $ \|Q_h w - w_h \| $ & $O(h^r)$ &  $ \3bar Q_h w- w_h\3bar_{W_h} $ & $O(h^r)$   \\ 
 3&    0.393E-04 &  4.0&    0.540E-04 &  2.6\\
 4&    0.505E-05 &  3.0&    0.150E-04 &  1.9\\
 5&    0.766E-06 &  2.7&    0.394E-05 &  1.9\\
 \hline 
 &  $ \|Q_h \btheta -\btheta_h \| $ & $O(h^r)$ &  $ \3bar Q_h \btheta- \btheta_h\3bar_{\Theta_h} $ & $O(h^r)$   \\  
 3&    0.470E-04 &  3.4&    0.584E-03 &  2.6\\
 4&    0.142E-04 &  1.7&    0.224E-03 &  1.4\\
 5&    0.388E-05 &  1.9&    0.112E-03 &  1.0\\ 
\hline 
\end{tabular} \end{center}  \end{table}

  \begin{table}[H]
  \caption{ Error profile for computing \eqref{s-1} by the $P_{2}$-$P_{2}$-$P_{2}$/$P_{2}$-$P_{2}$-$P_{5}$ element on polygonal
        meshes shown in Figure \ref{f-3}.} \label{t14}
\begin{center}  
   \begin{tabular}{c|rr|rr}  
 \hline 
&\multicolumn{4}{c}{ $t=1$ in \eqref{m1}.  }\\ \hline
$G_i$ &  $ \|Q_h w - w_h \| $ & $O(h^r)$ &  $ \3bar Q_h w- w_h\3bar_{W_h} $ & $O(h^r)$   \\  
 3&    0.355E-04 &  6.2&    0.103E-02 &  4.0\\
 4&    0.162E-05 &  4.5&    0.742E-04 &  3.8\\
 5&    0.162E-06 &  3.3&    0.158E-04 &  2.2\\ 
 \hline 
 &  $ \|Q_h \btheta -\btheta_h \| $ & $O(h^r)$ &  $ \3bar Q_h \btheta- \btheta_h\3bar_{\Theta_h} $ & $O(h^r)$   \\ 
 3&    0.133E-04 &  5.6&    0.295E-03 &  3.5\\
 4&    0.907E-06 &  3.9&    0.392E-04 &  2.9\\
 5&    0.106E-06 &  3.1&    0.953E-05 &  2.0\\ 
 \hline 
&\multicolumn{4}{c}{ $t=10^{-2}$ in \eqref{m1}.  }\\
 \hline  
$G_i$ &  $ \|Q_h w - w_h \| $ & $O(h^r)$ &  $ \3bar Q_h w- w_h\3bar_{W_h} $ & $O(h^r)$   \\ 
 3&    0.800E-05 &  5.8&    0.364E-04 &  4.2\\
 4&    0.484E-06 &  4.0&    0.238E-05 &  3.9\\
 5&    0.318E-07 &  3.9&    0.172E-06 &  3.8\\
 \hline 
 &  $ \|Q_h \btheta -\btheta_h \| $ & $O(h^r)$ &  $ \3bar Q_h \btheta- \btheta_h\3bar_{\Theta_h} $ & $O(h^r)$   \\  
 3&    0.129E-04 &  5.4&    0.295E-03 &  3.5\\
 4&    0.863E-06 &  3.9&    0.391E-04 &  2.9\\
 5&    0.101E-06 &  3.1&    0.946E-05 &  2.0\\
\hline 
\end{tabular} \end{center}  \end{table}

  \begin{table}[H]
  \caption{ Error profile for computing \eqref{s-1} by the $P_3$-$P_3$-$P_3$/$P_3$-$P_3$-$P_6$ element on polygonal
        meshes shown in Figure \ref{f-3}.} \label{t15}
\begin{center}  
   \begin{tabular}{c|rr|rr}  
 \hline 
&\multicolumn{4}{c}{ $t=1$ in \eqref{m1}.  }\\ \hline
$G_i$ &  $ \|Q_h w - w_h \| $ & $O(h^r)$ &  $ \3bar Q_h w- w_h\3bar_{W_h} $ & $O(h^r)$   \\  
 2&    0.242E-02 &  5.6&    0.211E-01 &  4.7\\
 3&    0.207E-04 &  6.9&    0.904E-03 &  4.5\\
 4&    0.302E-06 &  6.1&    0.287E-04 &  5.0\\
 \hline 
 &  $ \|Q_h \btheta -\btheta_h \| $ & $O(h^r)$ &  $ \3bar Q_h \btheta- \btheta_h\3bar_{\Theta_h} $ & $O(h^r)$   \\ 
 2&    0.688E-03 &  5.5&    0.429E-02 &  4.6\\
 3&    0.470E-05 &  7.2&    0.133E-03 &  5.0\\
 4&    0.931E-07 &  5.7&    0.575E-05 &  4.5\\
 \hline 
&\multicolumn{4}{c}{ $t=10^{-2}$ in \eqref{m1}.  }\\
 \hline  
$G_i$ &  $ \|Q_h w - w_h \| $ & $O(h^r)$ &  $ \3bar Q_h w- w_h\3bar_{W_h} $ & $O(h^r)$   \\ 
 2&    0.165E-03 &  6.5&    0.767E-03 &  5.1\\
 3&    0.133E-05 &  7.0&    0.228E-04 &  5.1\\
 4&    0.478E-07 &  4.8&    0.764E-06 &  4.9\\
 \hline 
 &  $ \|Q_h \btheta -\btheta_h \| $ & $O(h^r)$ &  $ \3bar Q_h \btheta- \btheta_h\3bar_{\Theta_h} $ & $O(h^r)$   \\  
 2&    0.623E-03 &  5.5&    0.426E-02 &  4.6\\
 3&    0.470E-05 &  7.0&    0.134E-03 &  5.0\\
 4&    0.946E-07 &  5.6&    0.576E-05 &  4.5\\
\hline 
\end{tabular} \end{center}  \end{table}

  \begin{table}[H]
  \caption{ Error profile for computing \eqref{s-2} by the $P_{1}$-$P_{1}$-$P_{1}$/$P_{1}$-$P_{1}$-$P_{4}$ element on polygonal
        meshes shown in Figure \ref{f-3}.} \label{t16}
\begin{center}  
   \begin{tabular}{c|rr|rr}  
 \hline 
&\multicolumn{4}{c}{ $t=1$ in \eqref{m1}.  }\\ \hline
$G_i$ &  $ \|Q_h w - w_h \| $ & $O(h^r)$ &  $ \3bar Q_h w- w_h\3bar_{W_h} $ & $O(h^r)$   \\  
 3&    0.390E-02 &  2.6&    0.246E-01 &  2.2\\
 4&    0.958E-03 &  2.0&    0.986E-02 &  1.3\\
 5&    0.243E-03 &  2.0&    0.468E-02 &  1.1\\
 \hline 
 &  $ \|Q_h \btheta -\btheta_h \| $ & $O(h^r)$ &  $ \3bar Q_h \btheta- \btheta_h\3bar_{\Theta_h} $ & $O(h^r)$   \\ 
 3&    0.246E-02 &  2.0&    0.192E-01 &  1.7\\
 4&    0.784E-03 &  1.6&    0.932E-02 &  1.0\\
 5&    0.214E-03 &  1.9&    0.471E-02 &  1.0\\
 \hline 
&\multicolumn{4}{c}{ $t=10^{-2}$ in \eqref{m1}.  }\\
 \hline  
$G_i$ &  $ \|Q_h w - w_h \| $ & $O(h^r)$ &  $ \3bar Q_h w- w_h\3bar_{W_h} $ & $O(h^r)$   \\ 
 3&    0.547E-03 &  2.2&    0.214E-02 &  1.7\\
 4&    0.141E-03 &  2.0&    0.647E-03 &  1.7\\
 5&    0.358E-04 &  2.0&    0.170E-03 &  1.9\\
 \hline 
 &  $ \|Q_h \btheta -\btheta_h \| $ & $O(h^r)$ &  $ \3bar Q_h \btheta- \btheta_h\3bar_{\Theta_h} $ & $O(h^r)$   \\  
 3&    0.203E-02 &  2.2&    0.188E-01 &  1.7\\
 4&    0.636E-03 &  1.7&    0.921E-02 &  1.0\\
 5&    0.169E-03 &  1.9&    0.468E-02 &  1.0\\
\hline 
\end{tabular} \end{center}  \end{table}

  \begin{table}[H]
  \caption{ Error profile for computing \eqref{s-2} by the $P_{2}$-$P_{2}$-$P_{2}$/$P_{2}$-$P_{2}$-$P_{5}$ element on polygonal
        meshes shown in Figure \ref{f-3}.} \label{t17}
\begin{center}  
   \begin{tabular}{c|rr|rr}  
 \hline 
&\multicolumn{4}{c}{ $t=1$ in \eqref{m1}.  }\\ \hline
$G_i$ &  $ \|Q_h w - w_h \| $ & $O(h^r)$ &  $ \3bar Q_h w- w_h\3bar_{W_h} $ & $O(h^r)$   \\  
 3&    0.289E-03 &  4.7&    0.772E-02 &  4.0\\
 4&    0.207E-04 &  3.8&    0.105E-02 &  2.9\\
 5&    0.229E-05 &  3.2&    0.261E-03 &  2.0\\
 \hline 
 &  $ \|Q_h \btheta -\btheta_h \| $ & $O(h^r)$ &  $ \3bar Q_h \btheta- \btheta_h\3bar_{\Theta_h} $ & $O(h^r)$   \\ 
 3&    0.228E-03 &  4.3&    0.522E-02 &  3.4\\
 4&    0.247E-04 &  3.2&    0.114E-02 &  2.2\\
 5&    0.298E-05 &  3.0&    0.293E-03 &  2.0\\
 \hline 
&\multicolumn{4}{c}{ $t=10^{-2}$ in \eqref{m1}.  }\\
 \hline  
$G_i$ &  $ \|Q_h w - w_h \| $ & $O(h^r)$ &  $ \3bar Q_h w- w_h\3bar_{W_h} $ & $O(h^r)$   \\ 
 2&    0.127E-02 &  3.3&    0.699E-02 &  3.9\\
 3&    0.115E-03 &  3.5&    0.642E-03 &  3.4\\
 4&    0.798E-05 &  3.8&    0.457E-04 &  3.8\\ 
 \hline 
 &  $ \|Q_h \btheta -\btheta_h \| $ & $O(h^r)$ &  $ \3bar Q_h \btheta- \btheta_h\3bar_{\Theta_h} $ & $O(h^r)$   \\  
 2&    0.442E-02 &  4.4&    0.551E-01 &  3.2\\
 3&    0.223E-03 &  4.3&    0.519E-02 &  3.4\\
 4&    0.221E-04 &  3.3&    0.112E-02 &  2.2\\
\hline 
\end{tabular} \end{center}  \end{table}

  \begin{table}[H]
  \caption{ Error profile for computing \eqref{s-2} by the $P_3$-$P_3$-$P_3$/$P_3$-$P_3$-$P_6$ element on polygonal
        meshes shown in Figure \ref{f-3}.} \label{t18}
\begin{center}  
   \begin{tabular}{c|rr|rr}  
 \hline 
&\multicolumn{4}{c}{ $t=1$ in \eqref{m1}.  }\\ \hline
$G_i$ &  $ \|Q_h w - w_h \| $ & $O(h^r)$ &  $ \3bar Q_h w- w_h\3bar_{W_h} $ & $O(h^r)$   \\  
 1&    0.130E+00 &  0.0&    0.170E+01 &  0.0\\
 2&    0.210E-02 &  6.0&    0.551E-01 &  4.9\\
 3&    0.379E-04 &  5.8&    0.179E-02 &  4.9\\
 \hline 
 &  $ \|Q_h \btheta -\btheta_h \| $ & $O(h^r)$ &  $ \3bar Q_h \btheta- \btheta_h\3bar_{\Theta_h} $ & $O(h^r)$   \\ 
 1&    0.123E+00 &  0.0&    0.945E+00 &  0.0\\
 2&    0.212E-02 &  5.9&    0.318E-01 &  4.9\\
 3&    0.399E-04 &  5.7&    0.116E-02 &  4.8\\
 \hline 
&\multicolumn{4}{c}{ $t=10^{-2}$ in \eqref{m1}.  }\\
 \hline  
$G_i$ &  $ \|Q_h w - w_h \| $ & $O(h^r)$ &  $ \3bar Q_h w- w_h\3bar_{W_h} $ & $O(h^r)$   \\ 
 1&    0.108E-01 &  0.0&    0.154E+00 &  0.0\\
 2&    0.427E-03 &  4.7&    0.926E-02 &  4.1\\
 3&    0.126E-04 &  5.1&    0.311E-03 &  4.9\\
 \hline 
 &  $ \|Q_h \btheta -\btheta_h \| $ & $O(h^r)$ &  $ \3bar Q_h \btheta- \btheta_h\3bar_{\Theta_h} $ & $O(h^r)$   \\  
 1&    0.123E+00 &  0.0&    0.945E+00 &  0.0\\
 2&    0.215E-02 &  5.8&    0.319E-01 &  4.9\\
 3&    0.414E-04 &  5.7&    0.116E-02 &  4.8\\
\hline 
\end{tabular} \end{center}  \end{table}


\begin{thebibliography}{99}

\bibitem{19}{\sc D. Arnold, F. Brezzi, D. Marini},  {\em A family of discontinuous Galerkin finite elements for the Reissner–Mindlin plate}, J. Sci. Comput. 22 (2005)
25–45.

\bibitem{1}{\sc D. Arnold, R. Falk},  {\em A uniformly accurate finite element method for the Reissner–Mindlin plate}, SIAM J. Numer. Anal. 26 (1989) 1276–1290.

 
\bibitem{20}{\sc D. Arnold, F. Brezzi, R. Falk, D. Marini},  {\em Locking-free Reissner–Mindlin elements without reduced integration}, Comput. Methods Appl. Mech.
Engrg. 96 (2007) 3660–3671.


\bibitem{2}{\sc K.J. Bathe, F. Brezzi},  {\em On the convergence of a four-node plate bending element based on Mindlin-Reissner plate theory and a mixed interpolation},
in: The Mathematics of Finite Elements and Applications, V (Uxbridge, 1984), Academic Press, London, 1985, pp. 491–503.

\bibitem{3}{\sc K.-J. Bathe, F. Brezzi},  {\em A simplified analysis of two plate bending elements – the MITC4 and MITC9 elements}, in: Numerical Techniques for
Engineering Analysis and Design, Vol. 1, Martinus Nijhoff, Amsterdam, 1987.

\bibitem{4}{\sc K.J. Bathe, F. Brezzi, S.W. Cho},  {\em The MITC7 and MITC9 plate bending elements}, Comput. Struct. 32 (1989) 797–841.

\bibitem{5}{\sc K.J. Bathe, E.N. Dvorkin},  {\em A four-node plate bending element based on Mindlin/Reissner plate theory and mixed interpolation}, Internat. J. Numer.
Methods Engrg. 21 (1985) 367–383.

\bibitem{6}{\sc K.-J. Bathe, E.N. Dvorkin},  {\em A formulation of general shell elements-the use of mixed interpolation of tensorial components}, Internat. J. Numer.
Methods Engrg. 22 (1986) 697–722.

\bibitem{7}{\sc F. Brezzi, M. Fortin},  {\em Numerical approximation of Mindlin-Reissner plates}, Math. Comp. 47 (1986) 151–158.


\bibitem{8}{\sc F. Brezzi, K. Bathe, M. Fortin},  {\em Mixed interpolated elements for Reissner–Mindlin plates}, Int. J. Numer. Methods Eng. 28 (1989) 1787–1801.


\bibitem{9}{\sc G. Brezzi, M. Fortin, R. Stenberg},  {\em Error analysis of mixed-interpolated elements for Reissner–Mindlin plates}, Math. Models Methods Appl. Sci.
1 (1991) 125–151.


\bibitem{10}{\sc L. Beirao da Veiga, A. Buffa, C. Lovadina, M. Martinelli, G. Sangalli},  {\em An isogeometric method for the Reissner–Mindlin plate bending problem},
Comput. Methods Appl. Mech. Engrg. 209/212 (2012) 45–53.
 


\bibitem{26}{\sc L. Beirao da Veiga, K. Lipnikov, G. Manzini},  {\em Convergence analysis of the high-order mimetic finite difference method}, Numer. Math. 113 (2009)
325–356.


\bibitem{27}{\sc L. Beirao da Veiga, F. Brezzi, A. Cangiani, G. Manzini, D. Marini, A. Russo},  {\em Basic principles of virtual element methods}, Math. Models Methods
Appl. Sci. 23 (2013) 199–214.


\bibitem{21}{\sc L. Beirao da Veiga, D. Mora, G. Rivera},  {\em Virtual elements for a shear-deflection formulation of Reissner–Mindlin plates}, preprint, arXiv:
1710.07330v1.


\bibitem{22}{\sc C. Chinosi},  {\em Virtual elements for the Reissner–Mindlin plate problem}, Numer. Methods Partial Differential Equations 34 (2018) 1117–1144.

 
\bibitem{28}{\sc B. Cockburn, J. Gopalakrishnan, R. Lazarov},  {\em Unified hybridization of discontinuous Galerkin, mixed, and continuous Galerkin methods for second
order elliptic problems}, SIAM J. Numer. Anal. 47 (2009) 1319–1365.

 \bibitem{pdwg3} {\sc W. Cao, C. Wang and J. Wang},  {\em An $L^p$-Primal-Dual Weak Galerkin Method for div-curl Systems}, Journal of Computational and Applied Mathematics, vol. 422, 114881, 2023.
 \bibitem{pdwg4}{\sc  W. Cao, C. Wang and J. Wang},  {\em An $L^p$-Primal-Dual Weak Galerkin Method for Convection-Diffusion Equations}, Journal of Computational and Applied Mathematics, vol. 419, 114698, 2023. 
 \bibitem{pdwg5}{\sc W. Cao, C. Wang and J. Wang},  {\em A New Primal-Dual Weak Galerkin Method for Elliptic Interface Problems with Low Regularity Assumptions}, Journal of Computational Physics, vol. 470, 111538, 2022.
   \bibitem{wg11}{\sc  S. Cao, C. Wang and J. Wang},  {\em A new numerical method for div-curl Systems with Low Regularity Assumptions}, Computers and Mathematics with Applications, vol. 144, pp. 47-59, 2022.
\bibitem{pdwg10}{\sc  W. Cao and C. Wang},  {\em New Primal-Dual Weak Galerkin Finite Element Methods for Convection-Diffusion Problems}, Applied Numerical Mathematics, vol. 162, pp. 171-191, 2021. 


\bibitem{11}{\sc R. Duran, E. Hernández, L. Hervella-Nieto, E. Liberman, R. Rodríguez},  {\em Error estimates for low-order isoparametric quadrilateral finite elements
for plates}, SIAM J. Numer. Anal. 41 (2003) 1751–1772.



\bibitem{12}{\sc R. Duran, L. Hervella-Nieto, E. Liberman, R. Rodriguez, J. Solomin},  {\em Approximation of the vibration modes of a plate by Reissner–Mindlin equations},
Math. Comp. 68 (228) (1999) 1447–1463.


\bibitem{13}{\sc R. Duran, E. Liberman},  {\em On mixed finite element methods for the Reissner–Mindlin plate model}, Math. Comp. 58 (1992) 561–573.


\bibitem{14}{\sc R. Duran, E. Liberman},  {\em On the convergence of a triangular mixed finite element method for Reissner–Mindlin plates}, Math. Models Methods
Appl. Sci. 6 (1996) 339–352.


\bibitem{15}{\sc R. Falk, T. Tu},  {\em Locking-free finite elements for the Reissner–Mindlin plate}, Math. Comp. 69 (2000) 911–928.

 

\bibitem{he}{\sc K. He, J. Chen, L. Zhang, and M. Ran}, {\em 
 A Stabilizer-Free Weak Galerkin Finite Element Method for the Darcy-Stokes Equations}, 
Int. J. Numer. Anal. Mod., vol. 21, pp. 459-475, 2024.

\bibitem{24}{\sc P. Hansbo, D. Heintz, M. Larson},  {\em A finite element method with discontinuous rotations for the Mindlin-Reissner plate model}, Comput. Methods
Appl. Mech. Engrg. 200 (2011) 638–648.

 
\bibitem{16}{\sc T.J.R. Hughes, T.E. Tezuyar},  {\em Finite elements based upon mindlin plate theory with particular reference to the four node blinear isoparamtric
element}, J. Appl. Mech. Eng. 48 (1981) 587–598.

 
\bibitem{ku} {\sc N. Kumar and B. Deka}, {\em Weak Galerkin Finite Element Methods for Parabolic Problems with L2
Initial Data}, 
Int. J. Numer. Anal. Mod., vol. 20, pp. 199-228, 2023.


   
  

 

  
\bibitem{23}{\sc C. Lovadina, D. Marini},  {\em Nonconforming locking-free finite elements for Reissner–Mindlin plates}, Comput. Methods Appl. Mech. Engrg. 195 (2006)
3448–3460.

\bibitem{li}{\sc D. Li, Y. Li and Z, Yuan}, {\em A New Weak Galerkin Method with Weakly Enforced Dirichlet Boundary Condition}, 
Int. J. Numer. Anal. Mod., vol. 20, pp. 647-666, 2023.
 


 

 \bibitem{wg14}{\sc D. Li, Y. Nie, and C. Wang},  {\em Superconvergence of Numerical Gradient for Weak Galerkin Finite Element Methods on Nonuniform Cartesian Partitions in Three Dimensions}, Computers and Mathematics with Applications, vol 78(3), pp. 905-928, 2019.  
  \bibitem{wg1} {\sc D. Li, C. Wang and J. Wang},  {\em An Extension of the Morley Element on General Polytopal Partitions Using Weak Galerkin Methods}, Journal of Scientific Computing, 100, vol 27, 2024.  
 \bibitem{wg2} {\sc D. Li, C. Wang and S. Zhang},  {\em Weak Galerkin methods for elliptic interface problems on curved polygonal partitions}, Journal of Computational and Applied Mathematics, pp. 115995, 2024. 
\bibitem{wg5} {\sc D. Li, C. Wang, J.  Wang and X. Ye},  {\em Generalized weak Galerkin finite element methods for second order elliptic problems}, Journal of Computational and Applied Mathematics, vol. 445, pp. 115833, 2024.
 \bibitem{wg6} {\sc D. Li, C. Wang, J. Wang and S. Zhang},  {\em High Order Morley Elements for Biharmonic Equations on Polytopal Partitions}, Journal of Computational and Applied Mathematics, Vol. 443, pp. 115757, 2024.
 \bibitem{wg7} {\sc D. Li, C. Wang and J. Wang},  {\em Curved Elements in Weak Galerkin Finite Element Methods}, Computers and Mathematics with Applications, Vol. 153, pp. 20-32, 2024.
\bibitem{wg8} {\sc D. Li, C. Wang and J. Wang},  {\em Generalized Weak Galerkin Finite Element Methods for Biharmonic Equations}, Journal of Computational and Applied Mathematics, vol. 434, 115353, 2023.
 \bibitem{pdwg1} {\sc D. Li, C. Wang and J. Wang},  {\em An $L^p$-primal-dual finite element method for first-order transport problems}, Journal of Computational and Applied Mathematics, vol. 434, 115345, 2023.
 \bibitem{pdwg2} {\sc D. Li and C. Wang},  {\em A simplified primal-dual weak Galerkin finite element method for Fokker-Planck type equations}, Journal of Numerical Methods for Partial Differential Equations, vol 39, pp. 3942-3963, 2023.
\bibitem{pdwg6}{\sc  D. Li, C. Wang and J. Wang},  {\em Primal-Dual Weak Galerkin Finite Element Methods for Transport Equations in Non-Divergence Form}, Journal of Computational and Applied Mathematics, vol. 412, 114313, 2022.
  \bibitem{wg13}{\sc  D. Li, C. Wang, and J. Wang},  {\em Superconvergence of the Gradient Approximation for Weak Galerkin Finite Element Methods on Rectangular Partitions}, Applied Numerical Mathematics, vol. 150, pp. 396-417, 2020.
\bibitem{lxz} {\sc B. Li, X. Xie and S. Zhang},  {\em  BPS preconditioners for a weak Galerkin finite
element method for 2D diffusion problems with strongly discontinuous
coefficients}, Computers \& Mathematics with Applications, 76(4), pp.701-724, 2018.  


 \bibitem{mu} {\sc L. Mu, J. Wang, and X. Ye}, {\em  Weak Galerkin finite element method for second-order elliptic
problems on polytopal meshes}, International Journal of Numerical Analysis and Modeling,
12 (2015), pp. 31-53.

\bibitem{RM} {\sc  L. Mu,   J. Wang,   \& X. Ye},  {\em   A Weak Galerkin Method for the Reissner–Mindlin Plate in Primary Form}, J Sci Comput, vol. 75, pp. 782–802, 2018. 
 
 



\bibitem{29}{\sc D. Pietro, A. Ern},  {\em Hybrid high-order methods for variable-diffusion problems on general meshes}, C. R. Math. 353 (2015) 31–34.


 
\bibitem{17}{\sc R. Stenberg, M. Suri},  {\em An hp error analysis of MITC plate elements}, SIAM J. Numer. Anal. 34 (1997) 544–568.

   \bibitem{wang}{\sc C. Wang},  {\em Auto-Stabilized Weak Galerkin Finite Element Methods on Polytopal Meshes without Convexity Constraints, 	arXiv:2408.11927}.  
 
 
   \bibitem{wg15}{\sc C. Wang},  {\em New Discretization Schemes for Time-Harmonic Maxwell Equations by Weak Galerkin Finite Element Methods}, Journal of Computational and Applied Mathematics, Vol. 341, pp. 127-143, 2018.  
 
   \bibitem{pdwg7}{\sc  C. Wang},  {\em Low Regularity Primal-Dual Weak Galerkin Finite Element Methods for Ill-Posed Elliptic Cauchy Problems}, Int. J. Numer. Anal. Mod., vol. 19(1), pp. 33-51, 2022.
 
 \bibitem{pdwg8}{\sc  C. Wang},  {\em A Modified Primal-Dual Weak Galerkin Finite Element Method for Second Order Elliptic Equations in Non-Divergence Form}, Int. J. Numer. Anal. Mod., vol. 18(4), pp. 500-523, 2021.

 
 \bibitem{pdwg13}{\sc  C. Wang},  {\em A New Primal-Dual Weak Galerkin Finite Element Method for Ill-posed Elliptic Cauchy Problems}, Journal of Computational and Applied Mathematics, vol 371, 112629, 2020.
 
  
  

\bibitem{25}{\sc R. Wang, L. Mu, X. Ye},  {\em A locking free Reissner–Mindlin element with weak Galerkin rotations}, Discrete Contin. Dyn. Syst. Ser. B 24 (2019)
351–361.

 
   
 \bibitem{pdwg11}{\sc  C. Wang and J. Wang},  {\em A Primal-Dual Weak Galerkin Finite Element Method for Fokker-Planck Type Equations}, SIAM Numerical Analysis, vol. 58(5), pp. 2632-2661, 2020.
 \bibitem{pdwg12}{\sc  C. Wang and J. Wang},  {\em A Primal-Dual Finite Element Method for First-Order Transport Problems}, Journal of Computational Physics, Vol. 417, 109571, 2020.
  
 \bibitem{pdwg14}{\sc  C. Wang and J. Wang},  {\em Primal-Dual Weak Galerkin Finite Element Methods for Elliptic Cauchy Problems}, Computers and Mathematics with Applications, vol 79(3), pp. 746-763, 2020. 
 \bibitem{pdwg15}{\sc  C. Wang and J. Wang},  {\em A Primal-Dual Weak Galerkin Finite Element Method for Second Order Elliptic Equations in Non-Divergence form}, Mathematics of Computation, Vol. 87, pp. 515-545, 2018.  
   
 \bibitem{wg17}{\sc C. Wang and J. Wang},  {\em Discretization of Div-Curl Systems by Weak Galerkin Finite Element Methods on Polyhedral Partitions}, Journal of Scientific Computing, Vol. 68, pp. 1144-1171, 2016.    
   \bibitem{wg19}{\sc C. Wang and J. Wang},  {\em A Hybridized Formulation for Weak Galerkin Finite Element Methods for Biharmonic Equation on Polygonal or Polyhedral Meshes}, International Journal of Numerical Analysis and Modeling, Vol. 12, pp. 302-317, 2015. 
 \bibitem{wg20}{\sc  J. Wang and C. Wang},  {\em Weak Galerkin Finite Element Methods for Elliptic PDEs}, Science China, Vol. 45, pp. 1061-1092, 2015.  
 \bibitem{wg21}{\sc C. Wang and J. Wang},  {\em An Efficient Numerical Scheme for the Biharmonic Equation by Weak Galerkin Finite Element Methods on Polygonal or Polyhedral Meshes}, Journal of Computers and Mathematics with Applications, Vol. 68, 12, pp. 2314-2330, 2014.  
 
   \bibitem{wg18}{\sc C. Wang, J. Wang, R. Wang and R. Zhang},  {\em A Locking-Free Weak Galerkin Finite Element Method for Elasticity Problems in the Primal Formulation}, Journal of Computational and Applied Mathematics, Vol. 307, pp. 346-366, 2016.   
 

 \bibitem{wg12}{\sc  C. Wang, J. Wang, X. Ye and S. Zhang},  {\em De Rham Complexes for Weak Galerkin Finite Element Spaces}, Journal of Computational and Applied Mathematics, vol. 397, pp. 113645, 2021.
 
 \bibitem{wg3} {\sc C. Wang, J. Wang and S. Zhang},  {\em Weak Galerkin Finite Element Methods for Optimal Control Problems Governed by Second Order Elliptic Partial Differential Equations}, Journal of Computational and Applied Mathematics, in press, 2024. 
 
 \bibitem{itera} {\sc C. Wang, J. Wang and S. Zhang},  {\em A parallel iterative procedure for weak Galerkin methods for second order elliptic problems}, International Journal of Numerical Analysis and Modeling, vol. 21(1), pp. 1-19, 2023.
 \bibitem{wg9} {\sc C. Wang, J. Wang and S. Zhang},  {\em Weak Galerkin Finite Element Methods for Quad-Curl Problems}, Journal of Computational and Applied Mathematics, vol. 428, pp. 115186, 2023.
   \bibitem{wy3655} {\sc J. Wang, and X. Ye}, {\em A weak Galerkin mixed finite element method for second-order elliptic problems}, Math. Comp., vol. 83, pp. 2101-2126, 2014.


  \bibitem{wg4} {\sc C. Wang, X. Ye and S. Zhang},  {\em A Modified weak Galerkin finite element method for the Maxwell equations on polyhedral meshes}, Journal of Computational and Applied Mathematics, vol. 448, pp. 115918, 2024. 


 \bibitem{wang-ye} {\sc J. Wang and X. Ye}, {\em The Basics of Weak Galerkin Finite Element Methods},  	arXiv:1901.10035.
 
 \bibitem{wg10}{\sc  C. Wang and S. Zhang},  {\em A Weak Galerkin Method for Elasticity Interface Problems}, Journal of Computational and Applied Mathematics, vol. 419, 114726, 2023. 

 \bibitem{ela}{\sc  C. Wang and S. Zhang},  {\em Auto-stabilized weak Galerkin methods for elasticity problems}, Journal of Computational and Applied Mathematics, vol. 477, 117199, 2026. 
  
  \bibitem{pdwg9}{\sc  C. Wang and L. Zikatanov},  {\em Low Regularity Primal-Dual Weak Galerkin Finite Element Methods for Convection-Diffusion Equations}, Journal of Computational and Applied Mathematics, vol 394, 113543, 2021.
 
 \bibitem{wg16}{\sc  C. Wang and H. Zhou},  {\em A Weak Galerkin Finite Element Method for a Type of Fourth Order Problem arising from Fluorescence Tomography}, Journal of Scientific Computing, Vol. 71(3), pp. 897-918, 2017.  

 

\bibitem{18}{\sc X. Ye},  {\em A rectangular element for the Reissner–Mindlin plate}, Numer. Methods Partial Differential Equations 16 (2000) 184–193.


 \bibitem{ye}{\sc  X. Ye and S. Zhang},  {\em  A stabilizer free weak Galerkin finite element method for the biharmonic equation on
polytopal meshes}, SIAM J. Numer. Anal., vol 58, No. 5, pp. 2572-2588, 2020.

 
 \bibitem{yang}{\sc L. Yang, W. Mu, H. Peng and X. Wang}, {\em The Weak Galerkin Finite Element Method for the Dual-Porosity-Stokes Model}, 
Int. J. Numer. Anal. Mod., vol. 21, pp. 587-608, 2024.

 

\bibitem{zhang}{\sc J. Zhang, 
F. Gao, and
J. Cui},  {\em Weak Galerkin Finite Element Method Based on POD for Nonlinear Parabolic Equations},
International Journal of Numerical Analysis and Modeling, Vol. 22 (2), pp. 157–177, 2025. 


   
 
 

 

 
 

  

 
 

\end{thebibliography}
\end{document}